%% file: main.tex
\documentclass[a4paper]{amsart}
\usepackage{graphicx}
\usepackage{hyperref}
\usepackage{wasysym}
\usepackage{pdfsync}
\usepackage{listings}
\usepackage{color}
\usepackage[T1]{fontenc}
\usepackage[utf8]{inputenc} 

\usepackage{amsmath, amssymb, amsfonts, amscd, amsthm}
\usepackage{upgreek}
\newcommand{\C}{\mathbf{C}}
\renewcommand{\r}{\mathbf{R}}
\renewcommand{\k}{\mathbf{k}}
\newcommand{\W}{\mathcal{W}}
\newcommand{\id}{\mathfrak{I}}

\newcommand{\z}{\mathbf{Z}}
\newcommand{\q}{\mathbf{Q}}
\newcommand{\f}{\mathbf{F}}
\newcommand{\ft}{\mathbf{F}_2}
\newcommand{\gr}{\operatorname{gr}}
\newcommand{\supp}{\operatorname{supp}}
\newcommand{\cha}{\omega}
\newcommand{\ourmap}{\zeta}
\newcommand{\ou}{\,\textnormal{\sc or}\,}
\newcommand{\G}{\mathcal{G}}
\newcommand{\U}{BO_\infty^N}

\newtheoremstyle{pedro}{}{}{\itshape}{}{\sc}{~--}{ }{\thmname{#1}\thmnumber{ #2}\thmnote{ (#3)}}

\newtheoremstyle{pedrodef}{}{}{}{}{\sc}{~--}{ }{\thmname{#1}\thmnumber{ #2}\thmnote{ (#3)}}

\theoremstyle{pedro}
\newtheorem{lem}{Lemma}[section]

\newtheorem{thm}[lem]{Theorem}

\newtheorem{prop}[lem]{Proposition}

\newtheorem{coro}[lem]{Corollary}

\theoremstyle{remark}

\newtheorem{rmk}[lem]{Remark}

\theoremstyle{pedrodef}

\newtheorem{ex}[lem]{Example}

\title[Milnor~$K$-theory and the graded representation ring]{Milnor~$K$-theory and\\ the graded representation ring
}
\author{Pierre Guillot and J\'an Min\'a\v{c}}

\address{
Pierre Guillot\\Universit\'{e} de Strasbourg\\
Institut de Recherche Math\'{e}matique Avanc\'{e}e\\
7~Rue Ren\'{e} Descartes\\
67084 Strasbourg, France}
\email{guillot@math.unistra.fr}

\address{
J\'an Min\'a\v{c}\\
Department of Mathematics, Middlesex College\\
The University of Western Ontario\\
London, Ontario, Canada, N6A 5B7
}
\email{minac@uwo.ca}

\numberwithin{equation}{section}

\begin{document}

\maketitle

\begin{abstract}
  Let~$F$ be a field, let~$G= Gal(\bar F/F)$ be its absolute Galois
  group, and let~$R(G, \k)$ be the representation ring of~$G$ over a
  suitable field~$\k$.  In this preprint we construct a ring
  homomorphism from the mod 2 Milnor $K$-theory $k_*(F)$ to the graded
  ring~$\gr R(G, \k)$ associated to Grothendieck's $\gamma
  $-filtration.  We study this map in particular cases, as well as a
  related map involving the~$W$-group~$\G$ of~$F$, rather than~$G$.
  The latter is an isomorphism in all cases considered.  

  Naturally this echoes the Milnor conjecture (now a theorem), which
  states that $k_*(F)$ is isomorphic to the mod 2 cohomology of the
  absolute Galois group~$G$, and to the graded Witt ring~$\gr W(F)$.

  The machinery developed to obtain the above results seems to have
  independent interest in algebraic topology. We are led to construct
  an analog of the classical Chern character, which does not involve
  complex vector bundles and Chern classes but rather real vector
  bundles and Stiefel-Whitney classes. Thus we show the existence of a
  ring homomorphism whose source is the graded ring associated to the
  corresponding~$K$-theory ring~$KO(X)$ of the topological space~$X$,
  again with respect to the~$\gamma $-filtration, and whose target is
  a certain subquotient of~$H^*(X, \ft)$.

  In order to define this subquotient, we introduce a collection of
  distinguished Steenrod operations. They are related to
  Stiefel-Whitney classes by combinatorial identities.
\end{abstract}

% \include{abstract}
% \tableofcontents

\input{intro}
\input{ideal}

\input{gradedrep}
\input{milnor}

%TO DO : what is really \id_BG when BG is profinite rather than finite...

\bibliography{myrefs}
\bibliographystyle{amsalpha}
\end{document}

%% file: intro.tex
\section{Introduction}\label{sec-intro}

The Milnor conjecture (\cite{milnorconjecture}), now a theorem by
Voevodsky (\cite{vov1} and \cite{vov2}), is the statement that a certain
map
\[ h_F \colon k_*(F) \longrightarrow H^*(F, \ft)  \]
is an isomorphism. Here we have written~$k_*(F)$ for mod~$2$
Milnor~$K$-theory, and~$H^*(F, \ft)$ for the Galois cohomology of~$F$
(that is, the cohomology of its absolute Galois group).  However, the
conjecture has a second part, proved by Orlov, Vishik and
Voevodsky (\cite{vovwitt}), according to which there is also an
isomorphism
\[ s \colon k_*(F) \longrightarrow \gr W(F) \, ,   \]
where~$W(F)$ is the Witt ring of~$F$, and where~$\gr W(F)$ denotes the
graded ring associated to the filtration by powers of the fundamental
ideal~$I$. In his original paper Milnor relates these two statements
by means of a commutative square:
\[ \begin{CD}
k_n(F) @>s>> \gr^n W(F) = I^n / I^{n+1} \\
@V{h_F}VV                @VV{w_{2^{n-1}}}V \\
H^n(F, \ft) @>{\times \, \ell(-1)^{t - n}}>> H^{2^{n-1}}(F, \ft) \,
. 
\end{CD}
\]
Here~$\ell(-1)$ is a certain distinguished element in Galois
cohomology, and the bottom map is multiplication by~$\ell(-1)^{t - n}$
where~$t= 2^{n-1}$; the map denoted by~$w_{2^{n-1}}$ is (an algebraic
variant of) a Stiefel-Whitney class. This shows for example that,
whenever~$h_F$ is injective and~$\ell(-1)$ is not a zero divisor,
then~$s$ is also injective (see Theorem 4.1 and Remark 4.2
in~\cite{milnorconjecture}).

The map~$w_{2^{n-1}}$ is {\em a priori} not part of a {\em ring}
homomorphisms from~$\gr W(F)$ into the cohomology ring. However, one
may at least put (writing~$|x|$ for the degree of~$x$):
\[ \id_F = \{ x \in H^*(F, \ft) : \ell(-1)^{2^{|x|-1} - |x|} \, x = 0 \} \,
. 
\]
It is clear that~$\id_F$ is an ideal in~$H^*(F, \ft)$, and thus one
gets a commutative square of maps of rings: 
\[ \begin{CD}
k_*(F) @>s>> \gr W(F) \\
@V{h_F}VV                @VVV \\
H^*(F, \ft) @>>> H^*(F, \ft)/ \id_F \,
. 
\end{CD}
\]

Motivated by this, and with a view towards an application in field
theory, the first question we address in the paper is one in algebraic
topology: given a topological space~$X$, is there always an
ideal~$\id_X$ within the cohomology~$H^*(X, \ft)$ which generalizes
the ideal~$\id_F$ in the case of fields? Before we state the
(positive) answer, we need to present another crucial player in this
game, which is to replace the Witt ring. Namely, we shall consider
the~$K$-theory of real vector bundles over~$X$, written~$KO(X)$, and
the so-called~$\gamma $-{\em filtration}: this is defined {\em via} a
general construction, due to Grothendieck, which applies to
any~$\lambda $-ring (see~\cite{atiyahtall}, \cite{fultonlang}). Our
result, involving the associated graded ring~$\gr KO(X)$ with respect
to this filtration, is the following.

\begin{thm}
There is a collection of Steenrod operations~$\theta_n$, for~$n \ge
1$, each of degree~$2^{n-1} - n$, with the following properties. For
any topological space~$X$, let~$\W^*(X)$ be the subring of~$H^*(X, \ft)$
generated by the Stiefel-Whitney classes of real vector bundles
over~$X$. If we put 
\[ \id_X = \{ x \in \W^*(X) : \theta_{|x|} x = 0 \} \, , 
\]
then~$\id_X$ is an ideal in~$\W^*(X)$. 

Moreover, there is an explicit map of graded rings
\[ \cha \colon \gr KO(X) \longrightarrow \W^*(X) / \id_X \, .   \]
\end{thm}

In the text, this consists of Corollary~\ref{coro-ideal-sw} and
Theorem~\ref{thm-chern-character}; in
Lemma~\ref{lem-theta-n-for-fields}, we eventually prove that~$\id_X$
is indeed a generalization of~$\id_F$. 

\begin{rmk}
It was brought to our attention after this paper had been completed
that the operations~$\theta_n$ were considered in a different context
in the work~\cite{franjou} by Benson and Franjou; they were also
considered by Kuhn in~\cite{kuhn}, and by Adams in~\cite{adams}. Even
our use of the letter~$\theta $, almost agreeing with Kuhn's use
of~$\Theta $, is coincidental.
\end{rmk}

We think of the map~$\cha$ as a analog of the classical Chern
character. It is worth pointing out that its construction involves, in
degree~$n$, the map~$w_{2^{n-1}}$ as above. To give a flavour of the
operations~$\theta_n$, let us indicate that~$\theta_1 = \theta_2 = 1$,
\[ \theta_3 = Sq^1 \, , \quad \theta_4 = Sq^{3} Sq^{1} + Sq^{4} \, ,
\quad \theta_5 = Sq^{7} Sq^{3} Sq^{1} + Sq^{8} Sq^{2} Sq^{1} \, , \]
\[ \theta_6 = Sq^{15} Sq^{7} Sq^{3} Sq^{1} + Sq^{16} Sq^{6} Sq^{3}
Sq^{1} + Sq^{16} Sq^{7} Sq^{3} + Sq^{16} Sq^{8} Sq^{2} \, ,  \]
and 
\begin{multline*}
\theta_7 = Sq^{31} Sq^{15} Sq^{7} Sq^{3} Sq^{1} + Sq^{32} Sq^{14}
Sq^{7} Sq^{3} Sq^{1}\\ + Sq^{32} Sq^{15} Sq^{7} Sq^{2} Sq^{1} + Sq^{32}
Sq^{16} Sq^{6} Sq^{2} Sq^{1} + Sq^{32} Sq^{16} Sq^{8} Sq^{1} \, .
\end{multline*}

We will compute the ideal in many cases, mostly for classifying spaces
of finite groups. We also treat the universal case of the space~$
(BO_{\infty})^N$, which produces relations in~$\id_X$ for any
space~$X$ easily. In all of our examples we have~$\W^*(X) = H^*(X,
\ft)$ (as is often the case with familiar spaces).

This Theorem finds the following purely algebraic applications. An
important example of~$\lambda $-ring is given by the representation
ring~$R(G, \k)$ of the finite group~$G$ over the field~$\k$. One may go
through Grothendieck's construction of the~$\gamma $-filtration in
this case, and consider again the graded ring~$\gr R(G, \k)$. Few
results are available about these, and the reason may be that the
early investigations of the~$\gamma $-filtration on a~$\lambda
$-ring~$K$ were strongly focused on~$K\otimes\q$; indeed under some
conditions one has an isomorphism~$K\otimes \q \cong \gr K \otimes
\q$, one of the highlights of the theory (for example see Theorem III
3.5 in~\cite{fultonlang}). By contrast, each graded piece~$\gr ^n R(G,
\k)$ is torsion when~$G$ is finite ($n \ge 1$), so that the ring~$\gr
R(G, \k)\otimes \q$ is not interesting.

After giving some elementary calculations, namely in the case when~$G$
is cyclic and~$\k$ is either algebraically closed or~$\k=\r$, we achieve
the computation of~$\gr R(G, \k)\otimes \ft$ when~$G$ is an elementary
abelian~$2$-group. This relies heavily on the map~$\cha$, and indeed
we are not aware of any other approach. More precisely, we use the map 
\[ R(G, \r) \longrightarrow KO(BG) \, ,   \]
obtained by associating to any representation~$r \colon G \to O_n$ the
vector bundle whose classifying map is~$Br \colon BG \to BO_n$; this
is a homomorphism of~$\lambda $-rings, so it is compatible with
the~$\gamma $-filtration and induces a homomorphism between the
associated graded rings. Combined with~$\cha$, this yields a useful
map~$\gr R(G, \r) \to \W^*(BG)/\id_{BG}$. When~$G$ is elementary
abelian, it is an isomorphism, though this is far from being always
true. 

Based on this, one can tackle many groups of small size, by
considering their elementary abelian subgroups, and in this fashion we
compute~$\gr R(D_4, \k)\otimes \ft$, where~$D_4$ is the dihedral group
of order~$8$. We would like to emphasize that our computations show,
contrary to a common belief, that the~$\gamma $-filtration is
reasonably well-behaved.

As our final topic, we return to Milnor's conjecture and propose a
variant. The following Theorem gave its name to this paper. Recall
that the~$W$-group of a field is a certain quotient~$\G$ of the Galois
group~$G$ (see~\cite{minacspira} and \S\ref{subsec-milnor-k-theory}
below), which is frequently much easier to study.

\begin{thm}
  Let~$F$ be any field, and let~$\k$ be any field of characteristic
  different from~$2$.  Let~$\bar F$ be
  the separable closure of~$F$, and let~$G = Gal(\bar F / F)$ be the
  absolute Galois group of~$F$. Then there is a map
\[ \ourmap \colon k_*(F) \longrightarrow  \gr R(G, \k)\otimes \ft \, .   \]
If the characteristic of~$F$ is not~$2$, and if~$\k$ possesses an
embedding into~$\r$, then there is a commutative square
\[ \begin{CD}
k_*(F) @>{\ourmap}>> \gr R(G, \k)\otimes \ft \\
@V{h_F}VV             @VV{\cha}V  \\
H^*(F) @>>> H^*(F) / \id_F \, . 
\end{CD}
\]
 If we call~$\G$
the~$W$-group of~$F$, then there is a map
\[ \ourmap \colon k_*(F) \longrightarrow  (\gr R(\G, \k)\otimes \ft)_{dec} \, .   \]
And if~$\k$ possesses an embedding into~$\r$, then there is a commutative
square 
\[ \begin{CD}
k_*(F) @>{\ourmap}>> (\gr R(\G, \k)\otimes \ft)_{dec} \\
@V{h_F}VV             @VV{\cha}V  \\
H^*(\G)_{dec} @>>> H^*(\G)_{dec} / \id_\G \, . 
\end{CD}
\]
\end{thm}

 Here we use the notation~$A^*_{dec}$ for the {\em
  decomposable part} of the graded algebra~$A^*$, that is, the
subalgebra generated by elements of degree~$1$. For example one has
famously
\[ H^*(G, \ft) = H^*(\G, \ft)_{dec} \, .   \]

The proof of this Theorem relies on an understanding of the rings~$\gr
R(D_4, \k)$ and $\gr R(\z/4, \k)$. The reason behind this is
that~$D_4$ and~$\z/4$ play a special role as small quotients of
the~$W$-group (for more on this and generally for an exploration of
small~$p$-quotients of the absolute Galois group,
see~\cite{minacefrat}).

Moreover, the computations we perform with~$\gr R(G, \k)$ throughout
the article are useful for studying the map~$\ourmap$ on examples. We
end up proving that the map
\[ k_*(F) \xrightarrow{~\ourmap~} (\gr R(\G, \r)\otimes \ft)_{dec}  \]
is an isomorphism when~$F$ is a finite field, or a local field, or a
real-closed field, or a global field, or when~$\ell(-1)$ is not a zero
divisor in~$H^*(F, \ft)$, and in several other cases. There is no
example yet for which this map is not an isomorphism. (This would also
hold with~$\k=\q$ instead of~$\k=\r$, and we shall argue that there is
essentially no other interesting choice for~$\k$).

The paper is organized as follows. In Section~\ref{sec-ideal}, we
present the operations~$\theta_n$ and use them to define the
ideal~$\id_X$. In Section~\ref{sec-gradedrep}, we construct the
``character''~$\cha$; we also use our new machinery in order to
compute the ring~$\gr R(G, \k)$ for several examples of finite
groups~$G$. Finally in Section~\ref{sec-milnor}, we return to
Milnor~$K$-theory and define the map~$\ourmap$.

\smallskip

\noindent {\em Acknowledgements.} The first-named author would like to
thank the Pacific Institute for Mathematical Sciences at Vancouver for
its warm welcome during the year 2010-2011, and the Centre National de
la Recherche Scientifique in France for the financial support that
made this visit possible. 

Both authors thank Sunil Chebolu and Marcus Zibrowius for stimulating
discussions, Vincent Franjou for bringing the paper~\cite{franjou} to
our attention, and Paul Goerss, Parimala and Burt Totaro for
encouraging words.

%% file: ideal.tex
\section{The canonical ideal}\label{sec-ideal}

In this section we prove that the mod~$2$ cohomology of any
topological space possesses a canonical subquotient. The latter will
be the target of a map from the graded~$K$-theory ring, as described
in the next section. Proofs are mainly combinatorial.

The {\em real} topological~$K$-theory of the topological space~$X$
will be denoted by~$KO(X)$. This is the Grothendieck ring of real
vector bundles over~$X$, which is different from the ring that is
sometimes called the ``representable~$K$-theory'' of~$X$ and defined
by
\[ \mathbf{KO}(X) = [X, \z \times BO] \, .   \]
While~$KO(X)$ and~$\mathbf{KO}(X)$ agree when~$X$ is a finite
CW-complex for example, there are cases, such as~$X = BG$ for a finite
group~$G$, for which these two rings are not isomorphic. We shall
briefly mention~$\mathbf{KO}(X)$ in the course of the proof of
Lemma~\ref{lem-sw-involves-only-reps}, while most of the paper deals
with~$KO(X)$. Also note that our results never refer to Atiyah's
definition of the ``real $K$-theory'', which is designed for spaces
with involutions.

The mod~$2$ cohomology~$X$ we write~$H^*(X)$.

\subsection{The operations~$\theta_n$}

Given a class~$t\in H^1(X) = [X, \r \mathbb{P}^\infty]$, one can
consider the corresponding line bundle~$L$ over~$X$. The first
Stiefel-Whitney class of~$L$, essentially by definition, is~$w_1(L)=
t$. Our first objective is to study an analogous relationship between
line bundles and Stiefel-Whitney classes, as follows.

\begin{thm} \label{thm-theta-n}
For each~$n\ge 1$, there is a Steenrod operation~$\theta_n$ with the
following property. Given a space~$X$ and classes~$t_1, t_2, \ldots,
t_n \in H^1(X)$, let~$L_i$ be the line bundle corresponding
to~$t_i$. Let~$\rho$ denote the virtual vector bundle over~$X$ defined by

\[ \rho = (L_1 - 1) (L_2 - 1) \cdots (L_n - 1) \, ,   \]
where~$1$ stands for the trivial line bundle. Then the Stiefel-Whitney
class~$w_i(\rho)$ is zero for~$1 \le i < 2^{n-1}$, while 
\[ w_{2^{n-1}} (\rho) = \theta_n(t_1 t_2 \cdots t_n) \, .   \]
\end{thm}

\begin{ex}
  As~$w_1(L-1)= w_1(L) = t$ we see that we may take the identity
  operation for~$\theta_1$. A quick calculation shows that the second
  Stiefel-Whitney class of~$(L_1 - 1)(L_2 - 1)$ is~$t_1 t_2$, so we
  may again take the identity for~$\theta_2$. On the other hand, the
  fourth Stiefel-Whitney class of~$(L_1 -1)(L_2-1)(L_3-1)$ is
\[ t_1^2 t_2 t_3 + t_1 t_2^2 t_3 + t_1 t_2 t_3^2 = Sq^1(t_1 t_2 t_3)
\, ,   \]
so that~$Sq^1$ can be taken for~$\theta_3$.
\end{ex}

The proof of Theorem~\ref{thm-theta-n} will occupy the rest of
this section. We start with a couple of lemmas.

\begin{lem} \label{lem-first-sw}
Let~$\rho \in KO(X)$, and let~$i$ be the least positive integer so
that~$w_i(\rho ) \ne 0$. Then~$i$ is a power of two.
\end{lem}

This is well-known, and follows from Wu's formula
(see~\cite{milnorchar}). 

\begin{lem} \label{lem-rho-tens-L-minus-one}
  Let~$\rho \in KO(X)$, and suppose that we may write~$\rho = E^+ -
  E^-$, where~$E^+$ and~$E^-$ are vector bundles of the same
  rank. Suppose also that~$w_i(\rho ) = 0$ for~$1 \le i <
  2^{n-1}$. Then for any line bundle~$L$, we have~$w_i(\rho (L-1)) =
  0$ for~$1 \le i < 2^{n}$.
\end{lem}

\begin{proof}
From the previous lemma, it is enough to prove that~$w_i(\rho (L-1)) =
0$ for~$1\le i \le 2^{n-1}$.

Here and elsewhere we shall use the total Stiefel-Whitney class, which
is the homomorphism 
\[ w_T \colon KO(X) \to 1 + T H^*(X) [[T]]  \]
defined by~$w_T(\rho ) = 1 + w_1(\rho )T + w_2(\rho )T^2 + \cdots
$. For example, when~$\rho $ is as in the statement of the Lemma, we
have~$w_T( \rho ) = w_T(E^+) w_T(E^-)^{-1}$, and from this we see that
the hypotheses imply that 
\[ w_i(E^+) = w_i(E^-) \quad\textnormal{for}\quad 1 \le i < 2^{n-1} \,
. \tag{*}
\]

We are interested in the Stiefel-Whitney classes of~$\rho (L - 1) = (L
E^+ + E^-) - (L E^- + E^+)$, and thus we wish to prove that 
\[ w_i(L E^+ + E^-) = w_i(L E^- + E^+) \, , \tag{**}  \]
for~$1 \le i \le 2^{n-1}$. Now, the Stiefel-Whitney classes of~$L E^+
+ E^-$ are given by evaluating certain universal polynomials, using
the classes of~$L$, $E^+$ and~$E^-$; moreover, these polynomials
depend only on the rank of the vector bundles involved. As a result,
it is clear from (*) that in degrees less than~$2^{n-1}$, we would obtain
the same result using~$LE^- + E^+$ instead. In other words, equation
(**) holds for~$1 \le i < 2^{n-1}$, and the only non-trivial calculation
happens for~$i=2^{n-1}$.

In this degree, we have 
\[ w_{2^{n-1}}(LE^+ + E^-) =  w_{2^{n-1}}(LE^+) + w_{2^{n-1}} (E^-)
+ R_{\pm} \, ,  \]
where the last term is 
\[ R_{\pm} = \sum_{p+q = 2^{n-1}} w_p(LE^+) w_q(E^-) \, .   \]
In this sum the indices~$p$ and~$q$ are positive; arguing as above, we
see that~$R_{\pm} = R_{\mp}$, where~$R_{\mp}$ is defined by
exchanging~$E^+$ and~$E^-$, that is 
\[ w_{2^{n-1}}(LE^- + E^+) =  w_{2^{n-1}}(LE^-) + w_{2^{n-1}} (E^+)
+ R_{\mp} \, .  \]
The Lemma will be proved if we can establish that~$w_{2^{n-1}}(LE^+) +
w_{2^{n-1}} (E^-) = w_{2^{n-1}}(LE^-) + w_{2^{n-1}} (E^+)$.

Let~$a= w_1(L)$. Then the~$2^{n-1}$-st Stiefel-Whitney class
of~$LE^\pm$ is given by
%% %
%% \[ w_{2^{n-1}} (LE^+) = a^{2^{n-1}} + w_1(E^+) a^{2^{n-1}-1} +
%% w_2(E^+) a^{2^{n-1} - 2} + \cdots + w_{2^{n-1}} (E^+) \, .   \]
%% %
%
\[ w_{2^{n-1}}(LE^\pm) = a^{2^{n-1}} + w_{2^{n-1}}(E^\pm) + P(a,
w_1(E^\pm), \ldots, w_{2^{n-1} - 1}(E^\pm)) \, ,   \]
where~$P$ is a polynomial. So, using (*) again, we see indeed that~$w_{2^{n-1}}(LE^+) +
w_{2^{n-1}} (E^-)$ is left unchanged when~$E^+$ and~$E^-$ are
exchanged. This concludes the proof.
\end{proof}

This Lemma gives us at once the easy part of
Theorem~\ref{thm-theta-n}: namely, proving that~$w_i(\rho ) = 0$
for~$1 \le i < 2^{n-1}$ is now achieved with a routine induction.

To go further, let us write~$\sigma_k$ for the~$k$-th symmetric
function in~$L_1$, \ldots, $L_n$, computed in the ring~$KO(X)$. We also
write~$\sigma_0 = 1$, the trivial line bundle, and we understand
that~$\sigma_k = 0$ for~$k > n$. Put
\[ E^{even} = \sum_{k ~\textnormal{even}
} \sigma_k \quad\textnormal{and}\quad E^{odd} = \sum_{k ~\textnormal{odd}
} \sigma_k \, . 
\]
The virtual bundle~$\rho $ as in the statement of the Theorem is then 
\[ \rho = (-1)^n(E^{even} - E^{odd}) \, .   \]
We have already established that~$w_i(E^{even}) = w_i(E^{odd})$ for~$1
\le i < 2^{n-1}$. Note that~$E^{even}$ and~$E^{odd}$ both have
rank~$2^{n-1}$, and that~$E^{even}$ contains a copy of the trivial
line bundle~$\sigma_0$; as a result, we finally have~$w_{2^{n-1}}(\rho
) =  w_{2^{n-1}}(E^{odd})$. 

We can make this quite explicit. Indeed, if we put 
\[ m_k = \prod_{1 \le i_1 < i_2 < \cdots < i_k \le n} (t_{i_1} +
t_{i_2} + \cdots + t_{i_k})  \, , \]
then it is readily seen that 
\[ w_{2^{n-1}}(E^{odd}) = \prod_{k ~\textnormal{odd}  
} m_k \, .  \]

The following Lemma gives an expansion of the right hand side. It is
crucial to the proof of the Theorem, and indeed can be considered to
lie at the core of the paper.

\begin{lem} \label{lem-key}
In the polynomial ring~$\ft[t_1, \ldots, t_n]$, one has the following
identity:
\[ \prod_{k ~\textnormal{odd} } m_k = \sum_{2^{r_1} + \cdots + 2^{r_n}
= 2^{n-1}} t_1^{2^{r_1}} t_2 ^{2^{r_2}} \cdots t_n ^{2^{r_n}} \, .  \]
\end{lem}

For example for~$n=3$ this gives 
\[ t_1 t_2 t_3 (t_1 + t_2 + t_3) =  t_1^2 t_2 t_3 + t_1 t_2^2 t_3 +
t_1 t_2 t_3^2 \, .   \]
\begin{proof}
Let~$P$ denote the polynomial on the right hand side, given as a
sum. We write~$P(t_1 \leftarrow t_i)$ for~$P(t_i, t_2, t_3, \ldots,
t_n)$.

Consider a fixed term of the sum defining~$P(t_1 \leftarrow
t_i)$. If~$r_1 \ne r_i$, then the sum will contain another term
with~$r_1$ and~$r_i$ swapped; but this will be the same term
because~$t_1 = t_i$ now, so they will cancel. The only terms
remaining, for~$r_1 = r_i$, will add up to 
$$ \sum_{\begin{array}{c}
2^{r_2} + \cdots + 2^{r_n} = 2^{n-1} \\
r_i \ge 1
\end{array}
} t_2^{2^{r_2}} \cdots t_n^{2^{r_n}} \, .  $$
The reason for~$r_i \ge 1$ is that~$2^{r_i}$ has really
absorbed~$2^{r_1} + 2^{r_i} = 2\cdot 2^{r_i} = 2^{r_i + 1}$. However,
here is an elementary observation: $2^{n-1} - 1$ cannot be written as
the sum of~$n-2$ powers of~$2$. It follows that the condition~$r_i\ge
1$ is in fact automatically satisfied (otherwise~$2^{r_i} = 1$ and we
have written~$1 + 2^{r_2} + \cdots + 2^{r_n} = 2^{n-1}$ which is
impossible). Erase the condition~$r_i \ge 1$ in the sum above: the
result now blatently is independent of~$i$. So~$P(t_1 \leftarrow t_i)
= P(t_1 \leftarrow t_j)$, for any pair of indices~$i, j$.

In particular, we see that 
\[ P(t_1 \leftarrow t_i + t_j) = P(t_1 \leftarrow t_i) + P(t_1
\leftarrow t_j) = 0 \, .   \]
As a result the polynomial~$P$ is divisible by~$(t_1 + t_i + t_j)$. Of
course the variable~$t_1$ can be replaced by any other, and likewise
we see that~$P$ is divisible by~$(t_i + t_j + t_k)$ for any
triple~$(i, j, k)$. It is also easy to continue, for example 
\[ P(t_1 \leftarrow t_{i_1} + t_{i_2} + t_{i_3} + t_{i_4}) = P(t_1
\leftarrow t_{i_1} + t_{i_2}  )  + P(t_1 \leftarrow t_{i_3} +
t_{i_4}) = 0+ 0 = 0 \, ,  \]
so~$P$ is divisible by~$(t_1 + t_{i_1} + t_{i_2} + t_{i_3} +
t_{i_4})$. Pursuing the calculations in this fashion, we see that~$P$
is divisible by all the terms of~$m_k$ as long as~$k$ is
odd. These terms are coprime in the ring~$\ft[t_1, \ldots, t_n]$,
so~$P$ is divisible by their product, and a comparison of the degrees
gives the result.
\end{proof}

We can finally describe~$\theta_n$. We are going to rely on Milnor's
description of the dual~$\mathcal{A}^*$ of the Steenrod algebra,
see~\cite{milnorsteenrod}.  Recall that

(1) $\mathcal{A}^*$ is polynomial on variables~$\xi_i$ in
degree~$2^{i} - 1$.

(2) For any space~$X$, there is a map {\em of rings}
$$ \lambda^* \colon H^*(X) \to H^*(X) \otimes \mathcal{A}^* \, ,  $$
such that, for any Steenrod operation~$\theta $ and element~$x \in H^*
X$, we can recover~$\theta x$ by evaluating~$\lambda^*(x)$ at~$\theta $.

(3) For~$X= B\z/2$, whose cohomology is~$\f[t]$, one has 
$$ \lambda^*(t) = \sum t^{2^i} \otimes \xi_i \, .  $$

This allows the computation of~$\lambda^*(t_1 \cdots t_n)$ in our
situation. If we define~$$Sq(i_1, i_2, \ldots, i_k)$$ to be the
Steenrod operation dual to~$\xi_1^{i_1} \cdots \xi_k^{i_k}$, we
may put
$$ \theta_n = \sum Sq(i_1, \ldots , i_k)  $$
where the sum runs over all the elements which have degree~$2^{n-1} -
n$. We have then~$\theta_n(t_1 \cdots t_n) = P$, where~$P$ is again
the right-hand side in the identity of Lemma~\ref{lem-key}. This Lemma
thus asserts that 
\[ \theta_n(t_1 \cdots t_n) = w_{2^{n-1}} (\rho ) \, ,   \]
which concludes the proof of Theorem~\ref{thm-theta-n}. 

\begin{rmk}
(a) The reader may compute~$\theta_n$ easily using the free computer
algebra system {\sc Sage}, simply by entering

\definecolor{gris}{rgb}{0.8,0.8,0.8}
\lstset{basicstyle= \footnotesize, frame=single, backgroundcolor= \color{gris}}
\begin{lstlisting}
sage: A= SteenrodAlgebra(2, 'milnor')
sage: theta= lambda n: sum(A.basis(2^(n-1)-n)).basis('serre-cartan')
\end{lstlisting}

\noindent Subsequent calls to {\tt theta(n)} will give the value of~$\theta_n$,
in terms of the~$Sq^k$'s. The examples given in the Introduction were
computed in this way.

(b) The operations~$\theta_n$ are not uniquely defined by the
requirement that they satisfy Theorem~\ref{thm-theta-n}. However with
the above definition, the operations coincide, as announced in the
Introduction, with those considered by Benson and Franjou
in~\cite{franjou}, by Kuhn in~\cite{kuhn}, and by Adams
in~\cite{adams}. Furthermore, we want to point out the following
alternative description. The Steenrod algebra~$\mathcal{A}$ is a Hopf
algebra, and is equipped with an antipode~$c\colon \mathcal{A} \to
\mathcal{A}$. It turns out that
\[ \theta_n = c( Sq^{2^{n-1} - n}) \, ,   \]
see \S7, Corollary 6 in~\cite{milnorsteenrod}. 

\end{rmk}

\subsection{The ideal}

Given any graded algebra~$A^*$, we write~$A^*_{dec}$ for the
subalgebra generated by the elements of degree~$1$. 

\begin{prop}
For any topological space~$X$, let us put 
\[ I_{dec} = \{ x \in H^*(X)_{dec} : \theta_{|x|} (x) = 0 \} \, . 
\]
Then~$I_{dec}$ is an ideal in~$H^*(X)_{dec}$.
\end{prop}

Here~$|x|$ is the degree of~$x$.

\begin{proof}
Let~$n=|x|$, and assume that~$\theta_n(x)=0$. It is enough to prove
that, for~$y$ of degree~$1$, we have~$\theta_{n+1}(xy) = 0$. 

The element~$x$ is a sum of products of~$n$ elements of degree~$1$,
and by applying Theorem~\ref{thm-theta-n} several times, we see that
there exists some virtual vector bundle~$\rho $ over~$X$ such
that~$\theta_n(x) = w_{2^{n-1}}(\rho )$, while~$w_i(\rho ) = 0$ for~$1
\le i < 2^{n-1}$. Here we rely on the observation that, if~$\rho_1$
and~$\rho_2$ are vector bundles such that the Stiefel-Whitney classes
of {\em both}~$\rho_1$ and~$\rho_2$ vanish in degrees less
than~$2^{n-1}$, then the same can be said of the sum~$\rho_1 +
\rho_2$, and moreover~$w_{2^{n-1}}(\rho_1 + \rho_2) =
w_{2^{n-1}}(\rho_1) + w_{2^{n-1}}(\rho_2)$. Our assumption on~$x$ implies
thus that~$w_{2^{n-1}}(\rho )= 0$, so that, by
Lemma~\ref{lem-first-sw}, we have actually~$w_i(\rho ) = 0$ for~$1 \le
i < 2^n$.

A similar reasoning with~$xy$ shows that~$\theta_{n+1}(xy) =
w_{2^n}(\rho (L-1))$, where~$L$ is the line bundle defined
by~$y$. Lemma~\ref{lem-rho-tens-L-minus-one} guarantees that this
Stiefel-Whitney class vanishes, thus establishing the Proposition. 
\end{proof}

We can make a stronger statement using the ``splitting
principle''. This classical result states that, given a finite number
of vector bundles~$E_1, \ldots, E_k$ over~$X$, one can find a
space~$Y$ with a map~$p\colon Y\to X$, such that

(1) the map~$p^* \colon H^*(X) \to H^*(Y)$ is injective, and

(2) each of~$E_1, \ldots, E_k$ splits as a sum of line bundles, when
pulled-back over~$Y$.
In particular, the classes~$p^*(w_i(E_j))$ all belong
to~$H^*(Y)_{dec}$. From this we obtain the following.

\begin{coro} \label{coro-ideal-sw}
Let~$\W^*(X)$ denote the subring of~$H^*(X)$ generated by
Stiefel-Whitney classes. Put 
\[ \id_X = \{ x \in \W^*(X) : \theta_{|x|}(x) = 0 \} \, . 
\]
Then~$\id_X$ is an ideal in~$\W^*(X)$.
\end{coro}

The canonical subquotient~$\W^*(X)/\id_X$ of~$H^*(X)$ will be one of
our chief interests in the rest of the paper. The following Lemma
states some of its easy but useful properties.

\begin{lem} \label{lem-preserves-injectivity}
Let~$f\colon Y \to X$ be any map. Then

(1) We have~$f^*(\W^*(X)) \subset f^*(\W^*(Y))$ and~$f^*(\id_X)
\subset f^*(\id_Y)$, so that there is an induced map~$ f^\sharp \colon
\W^*(X)/\id_X \to \W^*(Y) / \id_Y$.

(2) If~$f^*$ is injective, so is~$f^\sharp$.
\end{lem}

\begin{proof}
Property (1) is obvious given the naturality of both Stiefel-Whitney
classes and Steenrod operations. We turn to (2).

Suppose that~$f^\sharp(x) = 0$, so that~$f^*(x) \in \id_Y$. By
definition, this means that~$\theta_n (f^*(x)) = 0$, where~$n$ is
the degree of~$x$. However, we have then~$f^*(\theta_n(x)) = 0$ so
that~$\theta_n(x) = 0$ by injectivity of~$f^*$, and we see that~$x\in
\id_X$. 
\end{proof}

In what follows we shall be particularly interested in the case~$X=
BG$, the classifying space of the group~$G$. In this situation we will
use notation such as~$H^*(G)$, $\W^*(G)$ and~$\id_G$. Here is a first
example.

\begin{prop} \label{prop-ideal-elementary-abelian} Let~$G= (\z/2)^r$
  be elementary abelian of rank~$r$. Then the cohomology ring~$H^*(G)
  =  \W^*(G) = \ft[t_1, \ldots, t_r] $ is polynomial, and the
  ideal~$\id_G$ is generated by the elements~$t_i^2t_j + t_i t_j^2$,
  for~$1 \le i, j \le r$.
\end{prop}

Of course the statement about~$H^*(G)$ is classical. We point out that
an equivalent formulation of the Proposition is that~$\id_G$ is
generated as an ideal by the kernel of~$Sq^1 = \theta_3$ viewed as a
map~$H^3(G)\to H^4(G)$.

\begin{proof}
We have 
\[ Sq^1( t_i^2 t_j + t_i t_j^2) = t_i^2 t_j^2 + t_i^2  t_j^2 = 0 \,
,   \]
so certainly~$t_i^2 t_j + t_i t_j^2 \in \id_G$, as~$\theta_3 =
Sq^1$. Let~$J$ be the ideal generated by these elements, and let us
prove that~$\id_G = J$. We have a succession of surjective maps 
\[ \ft[t_1, \ldots, t_r] \longrightarrow H^*(G)/J \longrightarrow
H^*(G)/\id_G \, .   \]

Fix an interger~$n$. It is clear that any two monomials in~$H^n(G)$
are equal modulo~$J$ as soon as they are written with the same
``alphabet'', ie if they involve exactly the same variables. In order
to exploit this, for any non-empty subset~$S$ of~$\{ 1, 2, \ldots, r \}$ of
cardinality~$\le n$ we pick a monomial~$x_S \in H^n(G)$ which involves
the variables~$t_i$ for~$i \in S$ and no other; we also arrange so
that only one variable is raised to a power greater
than~$1$. Concretely, if~$S = \{ i_1, \ldots, i_s \}$, we may
take~$x_S = t_{i_1}^{n-s+1} t_{i_2} \cdots t_{i_s}$. We have observed
that the elements~$x_S$, for such subsets~$S$, generate~$H^n(G)/J$, and
we shall prove now that they are linearly independent even
in~$H^n(G)/\id_G$. It will follow that these elements form a basis of
both~$H^n(G)/J$ and~$H^n(G)/\id_G$, and this being true for all~$n$,
we are compelled to conclude that~$\id_G = J$.

Examining the definitions, we see that we must establish that the
elements~$\theta_n(x_S)$ are linearly independent
in~$H^{2^{n-1}}(G)$. So let us assume that we are given a zero linear
combination, say 
\[ \sum_S \, \lambda_S\, \theta_n (x_S) = 0 \, .  \tag{*} \]
Fix a~$T$, and let us prove that~$\lambda_T = 0$.  The identity (*)
takes place in a polynomial ring, so it is possible for us to set~$t_i
= 0$ for~$i \not\in T$. From Lemma~\ref{lem-key}, we see that~$x_S$
and~$\theta_n(x_S)$ are written in the same ``alphabet'', and so the
new identity we obtain only involves those subsets~$S$ such that~$S
\subset T$. What is more, the element~$\theta_n(x_T)$ is the only one
whose monomials involve {\em all} the variables~$t_i$ for~$i\in T$. It
is thus clear that we have~$\lambda_T \theta_n(x_T) = 0$.

It remains to prove that~$\theta_n(x_T) \ne 0$. From the construction
of~$x_T$, we know that there is some variable, say~$t_{j}$, which
appears in~$x_T$ to the power one. Now set~$t_i = 1$ for~$i\in T$
if~$i \ne j$ (while maintaining~$t_i= 0$ for~$i\not\in T$), and let us
prove that~$\theta_n(x_T)$ is non-zero even then. Indeed, from the
factorized expression for~$\theta_n(x_T)$ appearing in
Lemma~\ref{lem-key}, we know that this element is a product of terms
of the form
\[ m_k = (t_{i_1} + \cdots + t_{i_k})  \]
where~$k$ is odd, the indices~$i_1, \ldots, i_k$ are in~$T$, and at
most one of them is~$j$. If there is no such index, then~$m_k = 1$
after the variables have been set to one; if there is such an
index, then~$m_k = t_j$. As a result~$\theta_n(x_T)$ is a power
of~$t_j$, and in particular it is not zero. It follows that~$\lambda_T
= 0$ and the proof is complete.
\end{proof}

\begin{rmk}
This Proposition can also be deduced from Proposition 5.10
in~\cite{kuhn}.  
\end{rmk}

\subsection{The universal example of~$\U$}

We will now compute the ``universal'' ideal~$\id_N:= \id_{\U} $ for the
space~$\U$ to be described next. This will give relations in
the cohomology of any space.

Consider first the classifying space~$BO_n$ of the~$n$-th compact
orthogonal group~$O_n$; its cohomology is 
\[ H^*(BO_n) = \ft[w_1, w_2, \ldots, w_n] \, .   \]
There are inclusions maps~$O_n \to O_{n+1}$, inducing maps~$BO_n \to
BO_{n+1}$, and we take~$BO_\infty$ to be the colimit of this
diagram. We have thus 
\[ H^*(BO_\infty) = \lim H^*(BO_n) = \ft[w_1, w_2, w_3, \ldots ] \, .   \]
Here we think of~$w_i$ as a generalized Stiefel-Whitney class, and
indeed Corollary~\ref{coro-ideal-sw} applies with~$\W^*(BO_\infty) =
H^*(BO_\infty)$, as we readily see.

Next, take an integer~$N$ and consider the cartesian product~$\U=
(BO_\infty)^N$. This space has projection maps~$p_i\colon \U \to
BO_\infty$, and we write~$w_j(p_i)$ for the class~$p_i^*(w_j)$. The
cohomology of~$\U$ is polynomial in all the classes~$w_j(p_i)$, for~$1
\le i \le N$ and~$j\ge 1$. In particular~$H^*(\U) = \W^*(\U)$ (in the
same generalized sense as above).

Given any space~$X$ and vector bundles~$E_1$, $E_2$, \ldots, $E_N$
over~$X$, we have classifying maps~$f_i \colon X \to BO_\infty$,
which we may combine into a map~$X \to \U$. The Stiefel-Whitney
classes of the bundle~$E_i$ are pulled back from the
classes~$w_j(p_i)$, so the relations in~$\id_N$ will yield relations
in~$\id_X$.

Let us now state the result. It is somehow easier to give a
presentation for~$H^*(\U)/\id_N$.

\begin{prop}
The ring~$H^*(\U)/\id_N$ is generated by the classes~$w_{2^j}(p_i)$,
for~$1\le i \le N$ and~$j \ge 0$. These are subject to the relations 
\[ w_{2^{j_1}}(p_{i_1}) \cdots w_{2^{j_s}}(p_{i_s}) =
w_{2^{k_1}}(p_{\ell_1}) \cdots w_{2^{k_t}}(p_{\ell_t}) \tag{*}\]
whenever the two sides of (*) have the same degree and involve exactly
the same variables.

Moreover, for any integer~$n$, let the binary expansion of~$n$ be 
\[ n = \sum_{k \ge 0} \, a_k 2^k \, .   \]
Then the class~$w_n(p_i)$ is given in~$H^*(\U)/\id_N$ by 
\[ w_n(p_i) = \prod_{k \ge 0} \, w_{2^k}(p_i)^{a_k} \, . \tag{**}  \]
\end{prop}

In other words, the ideal~$\id_N$ is generated by the relations (*)
and (**). It follows that, for any vector bundle~$E$ over any
space~$X$ whatsoever, we will always have such relations as~$w_5(E) =
w_1(E) w_4(E)$ modulo~$\id_X$ (relation of type (**)); and if~$F$ is
any other vector bundle over~$X$, the relation~$w_1(E) w_2(F)^2 =
w_1(E)^3 w_2(F)$ modulo~$\id_X$ always holds (relation of type (*)).

Let us also state a general consequence.

\begin{coro} \label{coro-bounded}
  Let~$X$ be a space such that~$\W^*(X)$ is generated by finitely many
  Stiefel-Whitney classes. Then there exists a constant~$C$ such that
\[ \dim_{\ft} \W^n (X)/ \id_X \le C \]
for all~$n\ge 0$.
\end{coro}

\begin{proof}[Proof of the Corollary]
  By the Proposition, there are no more monomials in the
  Stiefel-Whitney classes in a given degree than ways of picking a
  subset of the set of variables.
\end{proof}

We turn to the proof of the Proposition. We provide details for the
case~$N=1$ only, for the general case simply involves more complicated
notation. Note that the cohomology of~$ BO_\infty$ agrees with that
of~$BO_m$ in degrees~$< m$, and it will be technically easier to work
with such a space~$BO_m$ with~$m$ large enough. So for the duration of
this proof, the symbol~$\infty$ will stand for a conveniently large
integer.

The group~$O_\infty$ has a distinguished elementary abelian
subgroup~$T = (\z/2)^\infty$, given by the diagonal matrices with
entries~$\pm 1$. Moreover the map 
\[ H^*(BO_\infty) \longrightarrow H^* T = \ft[t_1, \ldots ,
t_\infty]  \]
is injective, and sends~$w_i$ to the~$i$-th symmetric function in
the variables~$t_1$, \ldots, $t_\infty$. From
Lemma~\ref{lem-preserves-injectivity}, the ideal~$\id_N$ is
the kernel of the map~$H^*(BO_\infty)\to H^* (T) / \id_T$.

The ring~$H^*(T)/\id_T$ is given by
Proposition~\ref{prop-ideal-elementary-abelian}. We see that in
degree~$n$, it has a basis in bijection with the non-empty subsets
of~$\{ 1, 2, \ldots, \infty \}$ of cardinality~$\le n$. Given such a
subset~$I$, we write~$t_{I, n}$ for the corresponding element. Note
that the multiplication is given by the rule 
\[ t_{I, n} t_{J, m} = t_{I\cup J, n+m} \, .   \]
We also put 
\[ t_{i, n} = \sum_{\#I = i} t_{I, n} \, .   \]
In this notation, the image of~$w_i$ is~$t_{i, i}$. We need to work
out the product of~$w_i$ and~$w_j$, and this is {\em a priori} given
by 
\begin{align*}
  t_{i, i} t_{j, j} & = \left(\sum_{\#I = i} t_{I, i} \right)
  \left(\sum_{\#J = j} t_{J, j} \right) \\
    & = \sum_{\#I= i, ~\#J= j} t_{I\cup J, i+j} \\
   & = \sum_{k\ge 1} \sum_{\#K = k} \alpha_{i, j, k} \, t_{K, i+j} \, . 
\end{align*}
This equality defines the number~$\alpha_{i, j, k}$, which is clearly
the number of ways of writing a set with~$k$ elements as a union~$I
\cup J$, where~$I$ has~$i$ elements and~$J$ has~$j$ elements. Luckily,
we are considering~$\alpha_{i, j, k}$ mod~$2$, and the formula above
can be drastically simplified.

\begin{lem} \label{lem-or}
Given~$i$ and~$j$, there is one and only one integer~$k$ such
that~$\alpha_{i, j, k}$ is odd. Moreover~$k$ is given by the following
recipe. Writing the binary expansions 
\[ i = \sum_{s\ge 0} a_{s} 2^s \quad\textnormal{and}\quad j = \sum_{s\ge 0} b_s 2^s \, , 
\]
then~$k$ is given by 
\[ k = \sum_{s\ge 0} \, \max(a_s, b_s) 2^s \, .   \]
This number will be written~$i \ou j$.
\end{lem}

We point out that the notation~$i \ou j$ is very common in computer
science; it reminds the reader that a given ``bit'' of~$i \ou j$ is
set to~$1$ when the corresponding bit of~$i$ is~$1$ {\em or} the
corresponding bit of~$j$ is~$1$. The operation~$i, j \mapsto i \ou j$
is commutative and associative.

\begin{proof}[Proof of Lemma~\ref{lem-or}]
For any integer~$n$ whose binary expansion is 
\[ n = \sum_{s \ge 0} \, a_s 2^s \, ,   \]
we note 
\[ \supp (n)= \{ s : a_s \ne 0 \} \, ,
\]
and call it the support of~$n$. The number~$i \ou j$ is characterized
as the only integer whose support is~$\supp(i)\cup \supp(j)$.

Elementary combinatorics reveal that, when~$\max(i, j)\le k \le i+j$,
the number~$\alpha_{i, j , k}$ is given by
\begin{align*}
\alpha_{i, j, k} & = {k \choose i} {i \choose k-j} \\
               & =  {k \choose j} {j \choose k-i} \\
               & = \frac{k!} {(k-i)! (k-j)! (i+j -k)!} \, . 
\end{align*}
(For other values of~$k$, we have~$\alpha_{i, j, k} = 0$ trivially.)
Now we rely on an elementary observation, which dates back to Kummer
(see~\cite{kummer}, pp 507-508):
\[ {n \choose m} ~\textnormal{is odd}~ \quad\Longleftrightarrow\quad \supp(m)
\subset \supp (n) \, . 
\]

It appears that~$\alpha_{i, j, k}$ is odd precisely when the
following two conditions are satisfied:

(a) $\supp(j) \subset \supp(k)$,

(b) $\supp(k-j) \subset \supp(i)$.

\noindent Note that whenever (a) holds, we have~$\supp(k-j)= \supp(k)
\smallsetminus \supp(j)$. From this it is clear that (a) and (b)
together are equivalent to the condition~$\supp(k) = \supp(i)\cup
\supp(j)$. 
\end{proof}

We now see that the following relation holds in~$H^*T/\id_T$: 
\[ t_{i,i} t_{j ,j} = t_{i \ou j, i+j} \, .   \]
If follows that 
\[ w_{i_1} w_{i_2} \cdots w_{i_k} = t_{i_1 \ou i_2 \ou \cdots \ou i_k,
~ i_1 + i_2 + \cdots + i_k} \, ,   \]
and that
\[ w_{i_1} w_{i_2} \cdots w_{i_k} = w_{j_1} w_{j_2} \cdots w_{j_\ell}   \]
whenever~$i_1 + \cdots + i_k = j_1 + \cdots + j_\ell$ and~$i_1 \ou
\cdots \ou i_k = j_1 \ou \cdots \ou j_\ell$. In particular, we have 
\[ w_{2^{j_1}} \cdots w_{2^{j_s}} = w_{2^{k_1}} \cdots w_{2^{k_t}}  \]
whenever the degrees on both sides are equal and 
\[ \{ 2^{j_1}, \ldots, 2^{j_s} \} = \{ 2^{k_1}, \ldots, 2^{k_t} \} \,
. 
\]
That is, the relations (*) hold. The relations (**) are equally clear
at this point.

Let~$J$ be the ideal generated by (*) and (**). To finish the proof,
we need to show that~$H^*(\U)/J$ injects into~$H^*(T)/\id_T$. However,
this is obvious, and follows from the fact that the elements~$t_{i,
  n}$ for different values of~$i$ are linearly independent.

%% file: gradedrep.tex
\section{Graded ~$K$-theory \& Graded representation rings}\label{sec-gradedrep}

We start this section by recalling the definition of the~$\gamma
$-filtration on a general~$\lambda $-ring. Next we consider, as first
examples, the representation rings of some finite groups. Then we
prove the existence of a ``character'', that is a ring homomorphism,
between the graded ring associated to the~$K$-theory of real vector
bundles over a topological space, and the subquotient of its
cohomology defined in the previous section.

\subsection{Grothendieck's construction}

Let us give some recollections about the~$\gamma $-filtration, which is
due to Grothendieck. Details may be found in~\cite{fultonlang},
\cite{atiyahtall}.   

The general setting is that of a~$\lambda $-ring~$K$ with
augmentation~$\varepsilon \colon K \to \z$ (in older terminology,
$\lambda $-rings were called ``special $\lambda $-rings''). So for
each~$n\ge 0$ there is a map~$\lambda^n \colon K \to K$, such that for
each~$x\in K$ one has~$\lambda^0(x)= 1$, ~$\lambda^1(x) = x$, and all
the identities presented in the references above. The kernel
of~$\varepsilon $ is denoted~$I$ and called the augmentation ideal.

The~$\gamma $-operations are then defined by
\[ \gamma ^n(x) = \lambda^n(x + n - 1) \, ,   \]
for~$x \in K$. For~$n\ge 1$, let~$\Gamma^n$ be the abelian subgroup
of~$K$ generated by all elements of the form
\[ \gamma^{k_1}(x_1) \gamma^{k_2}(x_2) \cdots \gamma^{k_s}(x_s)
\quad\textnormal{with}\quad \sum k_i \ge n \, , 
\]
where each~$x_i$ belongs to~$I$. One checks
that each~$\Gamma^n$ is an ideal in~$K$, that~$\Gamma^{n+1} \subset
\Gamma^n$, and that~$\Gamma^1 = I$. Writing~$\Gamma^0$ for~$K$ itself,
one can then consider the associated graded ring 
\[ \gr K = \Gamma^0/\Gamma^1 \oplus \Gamma^1 / \Gamma^2 \oplus
\Gamma^2 / \Gamma^3 \oplus \cdots \, , \]
where~$\Gamma^0 / \Gamma^1 = K/I \cong \z$.

Now for~$x\in K$ and~$i\ge 1$ we write~$c_i(x)$ for the class
of~$\gamma^i( x - \varepsilon (x))$ in~$\Gamma^i / \Gamma^{i+1}$. The
properties of the operations~$\lambda^i$ on~$K$ imply the simple
statement that those classes~$c_i(x)$ satisfy all the familiar axioms
of {\em Chern classes}; and indeed we shall call them the algebraic
Chern classes. This may be the place to point out that, in spite of
the many references to algebraic topology in this paper, {\em we will
  never mention the topological Chern classes}, and a notation such
as~$c_i(x)$ is always understood in the graded ring associated to
a~$\lambda $-ring.

The typical example of~$\lambda $-ring for us will be the~$K$-theory
of real vector bundles over a topological space on the one hand, and
the representation ring~$R(G, \k)$ of a finite group~$G$ over a
field~$\k$ on the other hand. When~$\k = \r$, we note that a real
representation of~$G$ defines a real vector bundle over the
classifying space~$BG$, thus yielding a map of~$\lambda $-rings
\[ R(G, \r) \longrightarrow KO(BG) \, ,   \]
which in turn induces a map 
\[ \gr R(G, \r) \longrightarrow \gr KO(BG) \, .  \]

\subsection{Elementary examples}

It is our intention to advertise the rings~$\gr R(G, \k)$, and encourage
further investigations by our readers. Computations with these have a
cohomological flavour, although the results are interestingly
different. Let us start with a couple of general statements.

\begin{lem} \label{lem-degree-1}
Let~$G$ be a finite group. Then the group~$\gr^1 R(G, \k)$ is
isomorphic to the group of~$1$-dimensional representations of~$G$
over~$\k$, under the tensor product operation. Moreover the isomorphism
is given by the first Chern class.
\end{lem}

\begin{proof}
This follows from Theorem 1.7, Chapter III, in~\cite{fultonlang}. 
\end{proof}

\begin{lem} \label{lem-torsion}
  Let~$G$ be finite, and assume that the characteristic
of~$\mathbf{k}$ is~$0$. Then for~$n\ge 1$, the group~$\gr^n R(G, \k)$ is
  torsion. More precisely, if the order of~$G$ is~$|G|$, then~$\gr^n
  R(G, \k)$ is killed by~$|G|^n$.
\end{lem}

Thus the ring~$\gr R(G, \k)\otimes \q$ is isomorphic to~$\q$ concentrated in degree~$0$. The early investigations of the Grothendieck filtration focused on~$\gr K\otimes \q$ (in particular in the case when~$K= KO(X)$, the~$K$-theory of an algebraic variety), which the Lemma shows is not interesting for~$K= \gr R(G, \k)$. This may account for the lack of attention paid to these rings so far. We shall give several examples showing that~$\gr R(G, \k)$ is far from trivial.

\begin{proof}
% (THIS PROOF WORKS FOR $k$ OF CHAR 0, I'M CONVINCED THERE MUST BE A
% GENERAL ARGUMENT)

We make use of the {\em Adams operations} $\Psi^k$,  which are defined in
any~$\lambda $-ring. When~$\chi $ is the character of a
representation, then~$\Psi^k \chi (g) = \chi (g^k)$, for~$g\in G$. So
for~$k=|G|$, we see that~$\Psi^{|G|} \chi = \chi (1)$ (copies of the
trivial representation). In particular~$\Psi^{|G|} \chi $ depends only
on the degree of~$\chi $, and so~$\Psi^{|G|} x = 0$ for~$x\in I$, the
augumentation ideal. {\em A fortiori}, the operation~$\Psi^{|G|}$ is
zero on~$\Gamma^n$.

However, we have for~$n\ge 1$ and~$x\in \Gamma^n$ the relation 
\[ \Psi^k x = k^n x \quad\textnormal{mod}\quad \Gamma^{n+1} \, , 
\]
see Proposition III, 3.1 in~\cite{fultonlang}. The result follows. 
\end{proof}

\begin{coro} \label{coro-gr-R-bounded}
  Let the number of irreducible representations of the finite
  group~$G$ over~$\k$ be~$c+1$. Then for~$n\ge 1$ the group~$\Gamma^n$
  is isomorphic to~$\z^c$ as an abelian group, and consequently the
  group~$\gr^n R(G, \k)$ is generated by~$c$ elements.
\end{coro}

\begin{proof}
The group~$R(G, \k)$ is isomorphic to~$\z^{c+1}$, so the result is
certainly true for~$\Gamma^1 = I$. Moreover, the fact
that~$\Gamma^n/\Gamma^{n+1}$ is torsion indicates that~$\Gamma^{n+1}$
has the same rank as~$\Gamma^n$.
\end{proof}

Let us start giving concrete calculations.

\begin{prop} \label{prop-cyclic-alg-closed}
Suppose~$G= C_N$ is cyclic of order~$N$, and assume that~$\k$ is
algebraically closed, of characteristic prime to~$N$. Then 
\[ \gr R(G, \k) = \frac{\z[c_1( \rho )]} {(N c_1(\rho ))} \, ,
\]
for some $1$-dimensional representation~$\rho $. 
\end{prop}

In particular for each~$n\ge 1$ we have~$\gr^n R(G, \k)= \z/N$,
generated by~$c_1(\rho )^n$. Thus the graded representation ring in
this case is abstractly isomorphic to the ring~$H^{2*}(G, \z)$.

\begin{proof}
The irreducible representations are~$1$, $\rho $, $\rho^2$, \ldots,
$\rho^{n-1}$. We have~$c_1(\rho^k) = k c_1(\rho )$
(Lemma~\ref{lem-degree-1}). From the definitions, we see that~$\gr
R(G, \k)$ is thus generated by~$c_1( \rho )$, and from the relation~$\rho^N
= 1$ we see that~$N c_1(\rho ) = 0$. We need to show that there are no
further relations. So we consider an integer~$d$ such that~$d c_1(\rho
)^n = 0$, and we will show that~$N$ divides~$d$. This will show
that~$\gr ^n R(G, \k)$ is no smaller than~$\z/N$.

It is an easy general fact that, when all the irreducible
representations of~$G$ are~$1$-dimensional, then~$\Gamma^n = I^n$,
the~$n$-th power of the augmentation ideal. Now, we have~$R(G, \k)=
\z[\rho ] /(\rho^N - 1)$, and it is easy to see that~$I$ is the ideal
generated by~$\rho - 1$, so that~$I^n$ is generated by~$(\rho -
1)^n$. 

Write~$x= c_1(\rho )= \rho -1$ for simplicity. The relation~$dx^n = 0$
in~$\Gamma^n / \Gamma^{n+1}$ lifts to~$dx^n = P(x) x^{n+1}$ in~$R(G,
\k)$ for a polynomial~$P$, and in turn we may write this in the form
\[ d x^n = P(x) x^{n+1} + Q(x) [ (1+x)^N - 1]  \]
in~$\z[x]$. If~$n > 1$ then we may look at the terms of degree~$1$ on
each side, and deduce that~$Q(0) = 0$. So we may divide the equation
by~$x$ and obtain a similar one with~$n$ replaced by~$n-1$. So we go
all the way down to~$n=1$. In this case 
\[ dx = P(x)x^2 + Q(x)[(1+x)^N - 1] \, ,   \]
and by looking at the terms of degree~$1$ we see that~$d= Q(0)N$,
which was what we wanted.
\end{proof}

Now let us see how changing the field~$\k$ affects the results.

\begin{prop} \label{prop-cyclic-real}
Let~$G= C_N$ be cyclic of order~$N$, and let~$\k= \r$. Then

(1) If~$N$ is odd, one has 
\[ \gr R(G, \r) \otimes \ft = \ft \, ,   \]
concentrated in degree~$0$.

(2) If~$N= 2m$ with~$m$ odd, then 
\[ \gr R(G, \r) \otimes \ft = \ft[c_1(\varepsilon )] \, ,   \]
for some~$1$-dimensional representation~$\varepsilon $.

(3) If~$N= 2m$ with~$m$ even, then 
\[ \gr R(G, \r) \otimes \ft = \frac{\ft[c_1(\varepsilon ), c_2(r
  )]} {(c_1(\varepsilon )^2)} \, , 
\]
where~$\varepsilon $ is~$1$-dimensional, and~$r $ is~$2$-dimensional.
\end{prop}

\begin{proof}
The case (1) follows from Lemma~\ref{lem-torsion}. Assume~$N= 2m$. The
case~$m=1$ was already considered in the previous Proposition,
since~$R(C_2, \r) = R(C_2, \C)$, so we assume~$m > 1$. 

In this proof we work with characters rather than representations. The
irreducible characters of~$G$ over~$\C$ are~$1$, $\rho $, $\rho^2$,
\ldots, $\rho ^{N-1}$ as above. Put~$r_k= \rho^k + \rho^{-k}$
and~$\varepsilon = \rho^m$; in particular~$r_0= 2$ (two copies of the
trivial character) and~$r_m= 2 \varepsilon $. The group~$G$ has~$m+1$
irreducible real characters, namely~$r_k$ for~$1 \le k \le m-1$, together
with~$\varepsilon $ and the trivial character. When talking about the
characters~$r_k$ the indices will always be understood modulo~$N$.

The relation~$c_1(r) = c_1(\det(r))$ always holds (it follows from the
axioms for Chern classes). Thus we see that~$c_1(r_k) = 0$. Moreover
one checks immediately that 
\[ r_k r_\ell = r_{k+\ell} + r_{k - \ell} \, .   \]
We have~$c_2(r_{k+\ell} + r_{k - \ell}) = c_2(r_{k+\ell}) + c_2(r_{k -
  \ell})$, while~$c_2(r_kr_\ell)= 2(c_2(r_k) + c_2(r_\ell)) = 0$
(since we have tensored with~$\ft$). We draw 
\[ c_2(r_{k+\ell}) = c_2( r_{k-\ell}) \, .   \]
For~$\ell=1$ this shows that~$c_2(r_k)$ depends only on the parity
of~$k$; for~$k$ even we thus have~$c_2(r_k) = c_2(r_0) = 0$, while
for~$k$ odd we have~$c_2(r_k) = c_2(r_1)$.

Now, we use the relation~$\varepsilon r_1 = r_{m+1}$, which upon
comparison of the second Chern classes gives 
\[ c_2(r_1) + c_1(\varepsilon )^2 = c_2(r_{m+1}) \, . \tag{*}  \]
Here we need to distinguish between cases (2) and (3).

If~$m$ is odd, then~$m+1$ is even and~$c_2(r_{m+1}) = 0$, so that
equation~(*) shows that~$c_2(r_1) = c_1(\varepsilon )^2$. In this case
we see that~$\gr R(G, \r) \otimes \ft$ is generated
by~$c_1(\varepsilon )$. To show that there are no relations,
consider the subgroup~$C_2 \subset C_N$ and the restriction map 
\[ \gr R(C_N, \r)\otimes \ft \longrightarrow  \gr R(C_2, \r)\otimes
\ft = \ft[c_1(\bar \varepsilon )] \, .   \]
The equality on the right-hand side is the case~$m=1$ already
considered. It is clear that the restriction map sends~$\varepsilon $
to~$\bar \varepsilon $ (in obvious notation), and so
also~$c_1(\varepsilon )$ to~$c_1( \bar \varepsilon )$. The result
follows.

Now consider the case when~$m$ is even, so~$m+1$ is odd and equation
(*) reads~$c_1(\varepsilon )^2 = 0$. We see that~$ \gr R(G, \r)\otimes
\ft$ is generated by~$c_1(\varepsilon )$ and~$c_2(r_1)$, subject
to~$c_1(\varepsilon )^2 = 0$, and we need to show that there are no
further relations. We claim that~$c_2(r_1)^n \ne 0$ for all~$n$, and
that~$c_1(\varepsilon ) c_2(r_1)^n \ne 0$ for all~$n$, which will
suffice. To prove the claim we consider the map 
\[ \gr R(G, \r) \longrightarrow \gr R(G, \C) = \frac{\z[c_1(\rho )]}
{(N c_1(\rho ))} \, , 
\]
the equality coming from the previous Proposition. The
complexification map sends~$r_1$ to~$\rho + \rho^{-1}$
and~$\varepsilon $ to~$\rho^m$; so it sends~$c_1(\varepsilon )$ to~$m
c_1(\rho )$ and~$c_2(r_1 )$ to~$-c_1(\rho )^2$. From this it is
immediate that~$c_2(r_1 )^n$ is not zero in~$\gr^{2n} R(G, \r)\otimes
\ft$, while for~$c_1(\varepsilon ) c_2(r_1)^n$ we see that it is not
zero at least in~$\gr^{2n+1} R(G, \r)$; tensoring with~$\ft$ will not
hurt, though, for the relation~$\varepsilon^2 = 1$ yields~$2
c_1(\varepsilon ) = 0$, so that~$\gr^{2n+1} R(G, \r)$ is~$2$-torsion
anyway.
\end{proof}

We see that over the field of real numbers, the ring~$\gr R(G, \r)$
seems to be related to~$H^*(G, \ft)$, rather than~$H^{2*}(G,
\z)$. This connection will be strengthened by the ``character'' which
we will now define.

\subsection{The character}

We think of the following map as a mod~$2$ version of the Chern
character. 

\begin{thm} \label{thm-chern-character}
For any topological space~$X$, there is a map of rings 
\[ \cha \colon \gr KO(X) \longrightarrow \W^*(X)/\id_X \, .   \]
Moreover~$\cha(c_i(\rho )) = w_i(\rho )$.
\end{thm}

The proof will occupy the rest of this section. Throughout, we
write~$\Gamma^n$ for the~$n$-th stage in the~$\gamma $-filtration
of the~$\lambda $-ring~$KO(X)$.

\begin{lem} \label{lem-w-well-defined}
The application
\[ w_{2^{n-1}} \colon \Gamma^n \longrightarrow H^{2^{n-1}}(X)  \]
vanishes on~$\Gamma^{n+1}$. Thus it induces a map 
\[ \gr^n KO(X) \longrightarrow H^{2^{n-1}}(X) \, ,   \]
and this map is a homomorphism.
\end{lem}

\begin{proof}
First, consider a class in~$\Gamma^{n+1}$ of the form 
\[ \rho = (L_1 - 1) \cdots (L_{n+1} - 1) \, ,   \]
where each~$L_i$ is a line bundle. Then~$w_{2^{n-1}}( \rho ) = 0$ by
Theorem~\ref{thm-theta-n} (applied for~$n+1$). 

In general, we appeal again to the splitting principle. An
element~$\rho \in \Gamma^{n+1}$ may be written as a sum of elements
of the form
\[  \gamma^{k_1} (E_1 - \varepsilon (E_1)) \cdots
\gamma^{k_s}(E_s - \varepsilon (E_s))  \]
with~$\sum k_i \ge n+1$. However we may find a space~$Y$ and a
map~$Y\to X$ such that all the vector bundles~$E_i$ involved in the
expression of~$\rho $ split as sums of line bundles when pulled-back
to~$Y$; and moreover it may be arranged that the induced map~$H^*(X)\to H^*(Y)$ is
injective. 

To avoid multiple subscripts, let us work with a single vector
bundle~$E$, splittting up as a sum~$E = L_1 + \cdots + L_{\varepsilon
  (E)}$ over~$Y$. Then~$\gamma^k( E - \varepsilon (E))$ is the~$k$-th
symmetric function in the elements~$L_i - 1 = \gamma^1(L_i - 1)$. From
this and the particular case just studied, it follows easily
that~$w_{2^{n-1}}(\rho ) = 0 \in H^{2^{n-1}}(Y)$, and so this class is
also zero in~$H^{2^{n-1}}(X)$.

That~$w_{2^{n-1}}$ is a homomorphism is already true on~$\Gamma^n$,
and follows from the fact that all the elements in this group have
vanishing Stiefel-Whitney classes in degrees less than~$2^{n-1}$.
\end{proof}

To understand the map just defined, it is sufficient to indicate its
effect on product of Chern classes: indeed~$\Gamma^n / \Gamma^{n+1}$
is generated by those, by definition.

\begin{lem} \label{lem-chern2sw}
When~$\sum k_i = n$, we have
\[ w_{2^{n-1}}( c_{k_1}(E_1) \cdots c_{k_s}(E_s)    ) = \theta_n(
w_{k_1}(E_1) \cdots w_{k_s} (E_s)  ) \, .   \]
\end{lem}

\begin{proof}
  Again we start with the case when each~$E_i$ is a line bundle. In
  this case~$c_k(E_i) = 0$ for~$k > 1$, while~$c_1(E_i) = E_i - 1$; so
  it suffices to show that
  \[ w_{2^{n-1}} ( (E_1-1) \cdots (E_n-1) ) = \theta_n( w_1(E_1)
  \cdots w_1(E_n) ) \, ,  \]
which is the statement of Theorem~\ref{thm-theta-n}.

Unsurprisingly, the general case follows from the splitting
principle. Since Chern classes and Stiefel-Whitney classes behave in
the same way when a vector bundle splits as a sum, the argument is
easy, and will be omitted.
\end{proof}

The proof of the Theorem is now easy. Given an element~$\rho \in
\Gamma^n/\Gamma^{n+1}$, the class~$w^{2^{n-1}}(\rho )$ is well-defined
by Lemma~\ref{lem-w-well-defined} , and is of the form~$\theta_n(x)$
for~$x\in \W^n(X)$. So we may set~$\cha( \rho ) = x \in \W^n(X) /
\id_X$. Lemma~\ref{lem-chern2sw} proves both that~$\cha$ is a
homomorphism of rings, and that its values on Chern classes are as
announced in the Theorem. This concludes the proof.

As a by-product of the proof, we have the following result. Part (1)
is probably well-known, but it is just as easy (and more convenient
for our readers) to establish it directly. Part (2) should be compared
with Corollary~\ref{coro-gr-R-bounded}. 

\begin{lem} \label{lem-sw-involves-only-reps}
Let~$G$ be a finite group. 

(1) The ring~$\W^*(G)$ is generated by Stiefel-Whitney classes of real
representations of~$G$.

(2) There exists a constant~$C$ such that
\[ \dim_{\ft} \W^n (G)/ \id_G \le C \]
for all~$n\ge 0$.
\end{lem}

\begin{proof}
There are maps 
\[ R(G, \r) \longrightarrow  KO(BG) \longrightarrow  \mathbf{KO}(BG) \, ,   \]
and the Stiefel-Whitney class~$w_i$ (as a map) can be factored
through~$\mathbf{KO}(BG)$. Now we use Atiyah and Segal's completion
Theorem (see~\cite{atiyahsegal}, Theorem 7.1), which states that the
above diagram induces an identification of~$\mathbf{KO}(BG)$ with
the~$I$-adic completion of~$R(G, \r)$, where~$I$ is the augmentation
ideal.  Under this identification, we see that any element~$\rho \in
\mathbf{KO}(BG)$ may be approximated by a virtual representation~$r$
up to an element in a high power of~$I$. Since~$I^n \subset \Gamma^n$,
we know from Lemma~\ref{lem-w-well-defined} that by taking~$n$ large
enough we can insure that~$w_i(\rho ) = w_i(r)$ in a convenient
range. This proves (1).
  
  Property (2) follows from (1) and Corollary~\ref{coro-bounded}
  applied to~$BG$, keeping in mind that~$G$ only has finitely many
  real representations up to isomorphism.
\end{proof}

\subsection{Applications}

Our applications will use the natural map 
\[ \gr R(G, \r) \longrightarrow \gr KO(BG) \, ,   \]
which, when composed with the character~$\cha$, induces the map 
\[ \gr R(G, \r) \otimes \ft \longrightarrow \W^*(G)/\id_G \, .   \]
We will still denote this map by~$\cha$.

\begin{ex} \label{ex-cyclic-over-reals} We start with a few examples
  for which the source and target of~$\cha$ can be computed separately
  with relative ease (yet they are relevant to our applications to
  Milnor~$K$-theory). The notation is as in
  Proposition~\ref{prop-cyclic-real}.  Let~$G = C_N$ be a cyclic
  group. The cohomology ring~$H^*(G, \ft)$ is well-known, and admits
  the same description as~$\gr R(G, \r)\otimes \ft$ as in
  Proposition~\ref{prop-cyclic-real}, except that Chern classes are to
  be replaced with Stiefel-Whitney classes. In particular~$\W^*(G) =
  H^*(G)$.

  When~$N$ is odd, the map~$\cha$ is an isomorphism for uninteresting
  reasons (both rings being trivial). When~$N= 2m$ with~$m$ odd, we
  have~$H^*(G, \ft) \cong H^*(C_2, \ft)$, induced by the inclusion~$C_2
  \subset C_N$; it follows from  Proposition~\ref{prop-ideal-elementary-abelian} that~$\id_G =
  (0)$. As a result $\cha \colon \gr R(G, \r)\otimes \ft \to H^*(G)$
  is again an isomorphism.

When~$N= 2m$ with~$m$ even however, we shall see that~$\cha$ is not an
isomorphism, even though~$\gr R(G, \r)\otimes \ft$ and~$H^*(G)$ are
abstractly isomorphic. Let us compute~$\id_G$ first. We
have~$Sq^1(w_2(r_1)) = w_1(r_1) w_2(r_1)$ by Wu's formula,
and~$w_1(r_1) = w_1(\det(r_1)) = 0$; as a result~$Sq^1 w_1(\varepsilon
) w_2(r_1) = 0$, and~$w_1(\varepsilon ) w_2(r_1)\in \id_G$. In
fact~$\id_G$ is generated by this element. Indeed consider the two
maps
\[ H^*(G) / (w_1(\varepsilon ) w_2(r_1)) \longrightarrow H^*(G)/\id_G
\longrightarrow H^*(C_2) = \ft[w_1( \bar \varepsilon )] \, .  \]
The first one is surjective, and we need to see that it is an
isomorphism. This is trivially the case in odd degrees, both groups
being zero; in degree~$2n$, the group~$H^{2n}(G) / (w_1(\varepsilon )
w_2(r_1))$ is generated by~$w_2(r_1)^n$. The second map is the
restriction map induced by the inclusion~$C_2 \subset G$, so it sends~$w_2(r_1)^n$
 to~$w_1(\bar \varepsilon )^{2n}$. Thus~$w_2(r_1)^n \ne 0$
 in~$H^*(G)/\id_G$, and our computation of~$\id_G$ is complete.
 
The map 
\[ \cha \colon \gr R(G, \r)\otimes \ft \longrightarrow H^*(G)/\id_G  \]
has the form 
\[ \frac{\ft[c_1(\varepsilon ), c_2(r_1)]} {(c_1(\varepsilon )^2)}
\longrightarrow \frac{\ft[w_1(\varepsilon ), w_2(r_1)]}
{(w_1(\varepsilon )^2, w_1(\varepsilon ) w_2(r_1))} \, , 
\]
with~$c_1(\varepsilon )\mapsto w_1(\varepsilon )$ and~$c_2(r_1)\mapsto
w_2(r_1)$, so its kernel is generated by~$c_1(\varepsilon )c_2(r_1)$.
\end{ex}

The character~$\cha$ will help us compute the graded representation
ring of elementary abelian~$2$-groups. 

\begin{prop}
  Let~$G= (\z/2)^r$. Over any field~$\k$ of characteristic different
  from~$2$, there are~$1$-dimensional representations~$\varepsilon_1,
  \ldots,\varepsilon_r$ such that the graded representation ring is
\[ \gr R(G, \k)\otimes \ft = \frac{\ft[c_1(\varepsilon_1), \ldots,
  c_1(\varepsilon_r)]} { (c_1(\varepsilon_i)^2 c_1(\varepsilon_j) +
  c_1(\varepsilon_i)c_1(\varepsilon_j)^2  )}
\]
for~$1 \le i, j \le r$.
\end{prop}

\begin{proof}
The representation ring~$R(G, \k)$ is the same over any field of
characteristic $\ne 2$, so we may as well take~$\k= \r$. We have~$R(G)=
\z[\varepsilon_1, \ldots, \varepsilon_r]/(\varepsilon_i^2 -1)$, where
each~$\varepsilon_i$ has dimension~$1$, so the graded representation
ring is generated by the classes~$c_1(\varepsilon_i)$. 

Let~$x_i= \varepsilon_i - 1$. The relation~$\varepsilon_i^2= 1$ shows
that~$x_i^2 = -2 x_i$, so that~$x_i^2 x_j = x_i x_j^2 = -2 x_i
x_j$. Given that~$c_1(\varepsilon_i)$ is the image of~$x_i$
in~$\Gamma^1 / \Gamma^2$, we certainly have the
relation~$c_1(\varepsilon_i)^2 c_1(\varepsilon_j) =
c_1(\varepsilon_i)c_1(\varepsilon_j)^2$.

To prove that there are no more relations, we use the
character~$\cha$, which maps the graded representation ring
into~$H^*(G)/\id_G$. The latter is presented in
Proposition~\ref{prop-ideal-elementary-abelian}, and we
have~$c_1(\varepsilon_i)\mapsto w_1(\varepsilon_i) = t_i$ in the
notation of that Proposition. It is immediate that~$\cha$ is an
isomorphism, and the proof is complete.
\end{proof}

There seems to be no elementary way (that is, not relying on the
character) to prove this Proposition. 

The example of elementary abelian groups is useful when studying
larger groups, as we now show with the example of the dihedral
group. Before stating the result, let us fix some notation. We
write~$D_4$ for the dihedral group of order~$8$. The representation
ring~$R(D_4, \k)$ is the same over any field of characteristic~$\ne 2$:
there is a unique irreducible representation of dimension~$2$ which we
call~$\Delta $, and four irreducible representation of dimension~$1$,
say~$r_1$, $r_2$, and~$r_3 = r_1 r_2 = \lambda^2(\Delta )$, together
with the trivial one. The cohomology ring is given by 
\[ H^*(D_4) = \frac{\ft[w_1(r_1), w_1(r_2), w_2(\Delta )]} {(w_1(r_1)
  w_1(r_2))} = \W^*(D_4)\, . 
\]
(Here the representations are taken over~$\r$ of course.)

\begin{prop} \label{prop-graded-D4}
  Over any field~$\k$ of characteristic~$\ne 2$, the graded
  representation ring of the dihedral group is given by 
\[ \gr R(D_4, \k)\otimes \ft =  \frac{\ft[c_1(r_1), c_1(r_2), c_2(\Delta )]} {(c_1(r_1)
  c_1(r_2),  ~c_1(r_1) c_2(\Delta ) + c_1(r_2) c_2(\Delta ))} \, .  \]

On the other hand, the ideal~$I_{D_4}$ is generated by~$w_1(r_1)
w_2(\Delta )$ and~$w_1(r_2) w_2(\Delta )$. The kernel of the
character~$\cha$ is generated by~$c_1(r_1) c_2(\Delta )$ and~$c_1(r_2)
c_2(\Delta )$.
\end{prop}

\begin{proof}
First we show that the relations stated actually hold in the graded
 representation ring. The class~$c_2(\Delta )$ is the image
 of~$\gamma^2(\Delta - 2)$ in~$\Gamma^2/\Gamma^3$. By definition we
 have
\begin{align*}
\gamma^2(\Delta - 2) & = \lambda^2(\Delta -1) \\
                     &= 1 - \lambda^1(\Delta ) + \lambda^2(\Delta ) \\
                     &= 1 - \Delta + r_1 r_2 \, . 
\end{align*}

On the other hand~$c_1(r_i)$ is the image of~$\gamma^1(r_i - 1) = r_i
- 1$. We compute, using the relation~$r_i \Delta = \Delta $, that
\begin{align*}
\gamma^2(\Delta - 2) \gamma^1(r_i - 1) & = r_1 + r_2 - r_1 r_2 - 1 \\
                           &= -(r_1-1)(r_2 - 1) \\
                           & = - \gamma^1(r_1-1)  \gamma^1( r_2 - 1)
                           \, . 
\end{align*}

Thus we see that~$\gamma^1(r_1-1) \gamma^1( r_2 - 1) \in \Gamma^3$,
so~$c_1(r_1) c_1(r_2) = 0$. On the other hand, we also observe
that~$\gamma^2(\Delta - 2) \gamma^1(r_i - 1)$ has the same expression
for~$i=1$ and~$i=2$, so that~$c_1(r_1) c_2(\Delta ) = c_1(r_2)
c_2(\Delta )$, as announced.

To show that there are no further relations, we study the restrictions
to various subgroups. First, there are two copies of~$C_2\times C_2$
in~$D_4$, and the restriction map has the form 
\[ \gr R(D_4, \k)\otimes \ft \longrightarrow \gr R(C_2\times C_2,
\k)\otimes \ft = \frac{\ft[a, b]} {(a^2 b + a b^2)} \, , 
\]
where the equality comes from the previous Proposition. For one copy
of~$C_2 \times C_2$, with a judicious choice of coordinates, we
have~$c_1(r_1)\mapsto a+b$, $c_1(r_2)\mapsto 0$, and~$c_2(\Delta
)\mapsto ab$. What is more, using a different subgroup, we obtain a
restriction map with the roles of~$r_1$ and~$r_2$ exchanged. This is
already enough to see that the elements~$c_1(r_1)^{2n}$,
$c_1(r_2)^{2n}$ and~$\Delta^n$ are linearly independent in~$\gr^{2n}
R(D_4, \k)\otimes \ft$. 

There is also a copy of~$C_4$ in~$D_4$, and this time the restriction
map looks like
\[ \gr R(D_4, \k)\otimes \ft \longrightarrow \gr R(C_4,
\k)\otimes \ft = \frac{\ft[x, y]} {(x^2)} \, , 
\]
where the equality was proved in
Example~\ref{ex-cyclic-over-reals}. Under this map we
have~$c_1(r_1)\mapsto x$, $c_1(r_2)\mapsto x$, and~$c_2(\Delta
)\mapsto y$. The three restriction maps combined allow us to see
that~$c_1(r_1)^{2n+1}$, $c_1(r_2)^{2n+1}$, and~$c_1(r_1) \Delta^n$ are
linearly independent in~$\gr^{2n+1} R(D_4, \k)$. It follows that the
relations are precisely as stated.

To compute the ideal~$I_{D_4}$, we use Wu's formula again, which
gives~$Sq^1 w_2(\Delta )= w_1(\Delta ) w_2(\Delta )= w_1(r_1 r_2)
w_2(\Delta ) = (w_1(r_1) + w_1(r_2)) w_2(\Delta )$. As a result~$$Sq^1(
w_1(r_1) w_2(\Delta )  ) = w_1(r_1)^2 w_2(\Delta ) + w_1(r_1)^2
w_2(\Delta ) + w_1(r_1 ) w_1(r_2) w_2(\Delta ) = 0 \, . $$ It follows
that~$w_1(r_1) w_2(\Delta ) \in I_{D_4}$, and likewise for~$w_1(r_2)
w_2(\Delta )$. To show that there are no further generators needed
for~$I_{D_4}$, we study the restrictions to the above subgroups. This
is left to the reader.
\end{proof}

%% file: milnor.tex
\section{Applications to Milnor~$K$-theory
}\label{sec-milnor}

\subsection{Technicalities on profinite groups}

When dealing with Galois theory, we shall encounter profinite
groups. Let us indicate our conventions.

Let~$G$ be profinite. A representation of~$G$ over~$\k$, by convention,
means a finite-dimensional~$\k$-vector space with an action of~$G$
which factors through a quotient~$G/U$ where~$U$ is open. Such a~$U$
is closed and of finite index, so that~$G/U$ is finite. It follows
that the category of representations of~$G$ is semi-simple.

The Grothendieck group of this category is what we call~$R(G, \k)$. As
an abelian group, it is free with a basis given by the irreducible
representations. The maps~$R(G/U, \k) \to R(G, \k)$ are thus injective,
when~$U$ is as above, and it follows easily that 
\[ R(G, \k) = \operatorname{colim}_U R(G/U, \k) \, .   \]
The colimit is taken over a directed set, since~$U\cap V$ is open
when~$U$ and~$V$ are; so we can think of it essentially as a union. It
is clear that~$R(G, \k)$ is a~$\lambda $-ring.

The ``line elements'' in~$R(G, \k)$, as in~\cite{fultonlang}, are
the~$1$-dimensional representations, which form the group~$Hom(G,
\k^\times)$ where continuous homomorphisms are meant (using the
discrete topology of~$\k^\times$). Lemma~\ref{lem-degree-1} applies to
profinite groups with the same proof. 

The cohomology group~$H^i(G, \ft)$ can be identified with~$\lim_U
H^i(G/U, \ft)$ where again~$U$ runs among the normal, open subgroups
of~$G$. It follows that the Steenrod algebra acts on~$H^*(G, \ft)$. 

Finally, we define~$\W^*(G)$ to be the subring of~$H^*(G, \ft)$
generated by the Stiefel-Whitney classes of representations of~$G$,
where ``representation'' is understood in the above sense. By 
Lemma~\ref{lem-sw-involves-only-reps}, (1), this coincides with the
definition of~$\W^*(G) = \W^*(BG)$ when~$G$ is finite. 

The ideal~$\id_G$ in~$\W^*(G)$ can now be defined exactly as before.

\subsection{Milnor~$K$-theory and the graded representation ring} \label{subsec-milnor-k-theory}
Let us briefly recall the definition of mod~$2$ Milnor $K$-theory
(using the notation which is classically employed for Milnor
$K$-theory itself). Let~$F$ be any field. First one defines~$k_1(F)$
to be the~$\ft$-vector space~$F^\times/ (F^\times)^2$, written
additively, and the letter~$\ell$ is used to denote the identity 
\[ \ell \colon F^\times/ (F^\times)^2 \longrightarrow k_1(F) \, ,  \]
so that~$\ell(ab)= \ell(a) + \ell(b)$. Then one considers the tensor
algebra~$T^*(k_1(F))$, and the ideal~$M$ generated by the Matsumoto
relations, that is~$\ell(a) \ell(b) = 0$ whenever~$a + b= 1$. The
algebra~$k_*(F) = T^*(k_1(F))/M$, which is commutative, is the mod~$2$ Milnor
$K$-theory of~$F$.

\begin{thm} \label{thm-map-K-to-gr} Let~$F$ be any field, and let~$\k$
  be any field of characteristic different from~$2$.
%, not  containing~$\sqrt{-1}$. 
Let~$\bar F$ be the separable closure of~$F$, and let~$G = Gal(\bar F
/ F)$ be the absolute Galois group of~$F$. Then there is a natural map
\[ \ourmap \colon k_*(F) \longrightarrow  \gr R(G, \k)\otimes \ft \, .   \]
If the characteristic of~$F$ is not~$2$, and if~$\k$ possesses an
embedding into~$\r$, then there is a commutative square
\[ \begin{CD}
k_*(F) @>{\ourmap}>> \gr R(G, \k)\otimes \ft \\
@V{h_F}VV             @VV{\cha}V  \\
H^*(G) @>>> H^*(G) / \id_G \, . 
\end{CD}
\]
\end{thm}

Here the map~$h_F$ is the one originally defined by Milnor. This map
is an isomorphism, as the Milnor conjecture, now a theorem by
Voevodsky, states. In particular~$H^*(G)$ is generated
by~$1$-dimensional classes, and thus by Stiefel-Whitney classes, so
that~$H^*(G) = \W^*(G)$. In any case it is trivially the case
that~$h_F$ takes its values in~$\W^*(G)$, even without assuming
knowledge of the Voevodsky theorem.

\begin{proof}
We begin by defining a map 
\[ \tilde\ourmap \colon F^\times/ (F^\times)^2 \longrightarrow Hom(G,
\k^\times) \, .   \]
If~$a \in F^\times$, we consider the field extension~$F[\sqrt{a}]$
(within the fixed separable closure~$\bar F$). Then for~$\sigma \in
G$, we have~$\sigma (\sqrt a) /\sqrt a = \pm 1$. Thus we define a continuous
homomorphism~$\tilde\ourmap(a) \colon G \to \k^\times$ by
setting~$\tilde\ourmap(a) (\sigma ) = \sigma (\sqrt a)/\sqrt a$. It is immediate
that~$\tilde\ourmap$ is a homomorphism, and depends only on the class
of~$a$ modulo squares.

Since~$Hom(G, \k^\times)$ is none other than the group
of~$1$-dimensional representations of~$G$ over~$\k$, under the tensor
product operation, we see from Lemma~\ref{lem-degree-1} that we may
in fact define a map 
\[ \ourmap \colon k_1(F) \longrightarrow \gr^1 R(G, \k)\otimes \ft    \]
by~$\ourmap = c_1 \circ \tilde\ourmap \circ \ell^{-1}$ (here~$c_1$ is the
first Chern class).

This extends to a map~$T^*(k_1(F)) \to \gr R(G, \k)\otimes \ft$, and to
factor it into a map on~$k_*(F)$ we need to show that the
elements~$\ell(a) \ell(b)$ are sent to~$0$ when~$a+b= 1$. Let us use
an elementary result from Galois theory: when~$a+b=1$, and when~$a$
and~$b$ are linearly independent in~$F^\times / (F^\times)^2$
(over~$\ft$), then the fields~$F[\sqrt{a}]$ and~$F[\sqrt{b}]$ are both
contained in a field~$E$ such that~$E/F$ is Galois with~$Gal(E/F)=
D_4$, the dihedral group of order~$8$. Moreover, the
homomorphisms~$\tilde\ourmap (a)$ and~$\tilde\ourmap (b)$ factor
through~$D_4$, and as representations of this group they correspond
to~$r_1$ and~$r_2$ in the notation of
Proposition~\ref{prop-graded-D4}. However this very Proposition states
that
\[ c_1(r_1) c_1(r_2) = 0 = \ourmap(\ell(a)) \ourmap(\ell(b)) \, .   \]
We also need to take care of the case~$a= b \ne 0$ in the vector
space~$F^\times / (F^\times)^2$. Again, elementary Galois theory tells
us that in this situation, the field~$F[\sqrt{a}]$ is contained in a
field~$E$ such that~$E/F$ is Galois with Galois group~$\z/4$. Suppose
first that~$\k$ does not contain~$\sqrt{-1}$. Then~$R(\z/4, \k)$ is
the same as~$R(\z/4, \r)$. Moreover~$\tilde \ourmap(a)$ factors through~$\z/4$
as the representation~$\varepsilon $, in the notation of
Proposition~\ref{prop-cyclic-real}. This Proposition states that
\[ c_1( \varepsilon )^2 = 0 = \ourmap(\ell(a))^2 \, ,   \]
as we needed. Now suppose that alternatively~$\k$ does
contain~$\sqrt{-1}$. Then~$R(\z/4, \k)$ is the same as~$R(\z/4, \C)$,
and $\tilde \ourmap (a)$ factors through~$\z/4$ as~$\rho^{\otimes 2}$ in the
notation of Proposition~\ref{prop-cyclic-alg-closed}. Thus 
\[ c_1( \rho^{\otimes 2}) = 2 c_1(\rho )  = 0
= \ourmap(\ell(a)) \, ,   \]
since we have tensored with~$\ft$ (note however that this class
squares to zero even without tensoring, from the same
Proposition). Thus~$\ourmap$ factors through~$k_*(F)$ in any case.

The commutativity of the square follows from the property~$\cha
(c_1(r)) = w_1(r)$ for any representation~$r$. Indeed, Milnor's map
sends~$\ell(a)$ to~$w_1( \tilde \ourmap(a))$ (when~$\k=\r$).
\end{proof}

Of course one may wonder whether the map~$\ourmap$ is an isomorphism,
just like~$h_F$ is, and for what choices of~$\k$. During the course of
the proof we have seen that there will be a non-trivial kernel
when~$\k$ contains~$\sqrt{-1}$ (in some cases, compare the examples
below), so we will focus our attention on other fields. In general, the
question is a hard one, and we will study a certain variant which
lends itself to computation more easily. Assume that~$F$ has
characteristic~$\ne 2$. Following~\cite{minacspira}, let~$F_q$ be the
quadratic closure of~$F$ (the compositum within~$\bar F$ of all finite
extensions of~$F$ whose degree is a power of~$2$). Let~$Q =
Gal(F_q/F)$, which is a quotient of~$G = Gal(\bar F/ F)$. Then define
\[ \G = Q / Q^4[Q^2, Q] \, .   \]
This group is called the~$W$-group of~$F$ because of its relation with
the Witt ring. We are interested in~$\G$ because it is much easier to
deal with than~$G$ itself; for example when~$F^\times/ (F^\times)^2$
is finite then~$\G$ is also finite. However, the cohomological
information of~$F$ is preserved, since 
\[ H^*(G) = H^*(\G)_{dec} \, .   \]
(See~\cite{ademminac}, Theorem 3.14.) We then think of the following
Theorem as an amendment to Theorem~\ref{thm-map-K-to-gr}.

\begin{thm} \label{thm-map-K-to-grW} Let~$F$ and~$\k$ be fields of
  characteristic different from~$2$, 
%with~$\k$ not containing~$\sqrt{-1}$.
and let~$\G$ be the~$W$-group of~$F$.  Then there is a natural map
\[ \ourmap \colon k_*(F) \longrightarrow  (\gr R(\G, \k)\otimes \ft)_{dec} \, .   \]
If~$\k$ possesses an embedding into~$\r$, then there is a commutative
square 
\[ \begin{CD}
k_*(F) @>{\ourmap}>> (\gr R(\G, \k)\otimes \ft)_{dec} \\
@V{h_F}VV             @VV{\cha}V  \\
H^*(\G)_{dec} @>>> H^*(\G)_{dec} / \id_\G \, . 
\end{CD}
\]
\end{thm}

\begin{proof}
  Since the extension of~$F$ corresponding to~$\G$ contains all the
  extensions with Galois group~$D_4$, $\mathbf{Z}/4$ or~$\mathbf{Z}/2$ (\cite{minacspira}, Corollary
  2.18), the same proof works {\em mutatis mutandis}.
\end{proof}

We are already capable of proving that there are very few
possibilities for~$\k$:

\begin{lem}
  Let~$\k$ have characteristic~$0$. 
\begin{itemize}
\item Either~$\k$ contains~$\sqrt{-1}$, and the inclusion~$\q(\sqrt{-1}) \to
  \k$ induces an isomorphism~$R(\G, \q(\sqrt{-1})) = R(\G, \k)$; in
  particular one has~$R(\G, \k) = R(\G, \C)$.

\item Or, $\k$ does not contain~$\sqrt{-1}$, and the inclusion~$\q \to
  \k$ induces an isomorphism~$R(\G, \q) = R(\G, \k)$; in particular one
  has~$R(\G, \k) = R(\G, \r)$.
\end{itemize}
\end{lem}

\begin{proof}
Every element in~$\G$ has order dividing~$4$.
\end{proof}

So we are essentially left with the cases~$\k=\C$ and~$\k=\r$. We have
already dismissed the choice~$\k=\C$.

\begin{lem} \label{lem-ourmap-iso-lowdegrees}
When~$\k=\q$ or~$\r$, the map 
\[ \ourmap \colon k_*(F) \longrightarrow  (\gr R(\G, \k)\otimes \ft)_{dec} \]
is always surjective, and is an isomorphism in degrees~$1$ and~$2$.
\end{lem}

\begin{proof}
The map~$\ourmap$ is an isomorphism in degree~$1$ by choice of~$\k$, as
already alluded to (cf Lemma~\ref{lem-degree-1}). The target of this
map is generated by elements of degree~$1$ by definition, so~$\ourmap
$ is surjective.

In degree~$2$, we consider the commutative square of
Theorem~\ref{thm-map-K-to-grW}. The map~$h_F$ is injective (in
degree~$2$ this is a theorem of Merkurjev's, which is used in the
general proof of Voevodsky-Rost). The map~$H^2(\G)_{dec} \to
H^2(\G)_{dec}/\id_\G$ is also injective, since~$\theta_2$ is the
identity. It follows that~$\ourmap$ is injective in degree~$2$.
\end{proof}

\subsection{Examples}

We fix the choice~$\k=\r$ for definiteness. We shall give a collection
of examples of fields~$F$, assumed to be of characteristic~$\ne 2$,
for which the map 
\[ \ourmap\colon k_*(F) \longrightarrow (\gr R(\G, \r) \otimes \ft)_{dec}  \]
is an isomorphism. In each case this will follow either directly from
Lemma~\ref{lem-ourmap-iso-lowdegrees} or from the fact that
the~$W$-group~$\G$ is a finite group whose graded representation ring
we have been able to compute. Note that there is yet no example of
field~$F$ such that~$\ourmap$ is not an isomorphism.

\subsubsection{$F$ finite.} In this case the group~$\G$ is~$\z/4$, so that both
the Galois cohomology and the Milnor~$K$-theory of~$F$ are
concentrated in degrees~$\le 2$; the same can thus be said of~$(\gr
R(\G, \r) \otimes \ft)_{dec}$, from the surjectivity statement in
Lemma~\ref{lem-ourmap-iso-lowdegrees}. The same Lemma shows
that~$\ourmap$ is an isomorphism. This is consistent with
Proposition~\ref{prop-cyclic-real}.

\subsubsection{$F$ real closed.} In this case~$\G = \z/2$, so~$H^*(\G)= \ft[t]$
with~$t$ of degree~$1$.  The map~$\ourmap$ is clearly an isomorphism,
by Lemma~\ref{lem-ourmap-iso-lowdegrees} and
Proposition~\ref{prop-cyclic-real} combined.

\subsubsection{$F$ a local field.} In this case~$k_*(F)$ is again
concentrated in degrees~$\le 2$, so~$\ourmap$ is an isomorphism. 

\subsubsection{$F$ a global field.} We use the commutative diagram 
\[ \begin{CD}
k_*(F) @>{\ourmap}>>  (\gr R(\G, \r) \otimes \ft)_{dec} \\
@VVV                 @VVV \\
\bigoplus_v k_*(F_v) @>{\oplus \ourmap_v}>> \bigoplus_v  (\gr R(\G_v, \r) \otimes \ft)_{dec}
\end{CD}
\]
Here the direct sums extend over all the real completions~$F_v$ of the
field~$F$, while~$\G_v$ is the~$W$-group of~$F_v$, and~$\ourmap_v$ has
an obvious meaning. The vertical map on the left is an isomorphism in
degrees~$\ge 3$ by Tate (see~\cite{milnorconjecture}, Appendix), and
the maps~$\ourmap_v$ are isomorphisms by the real case already
considered, so~$\ourmap$ is injective, as well as surjective, in those
degrees. Finally~$\ourmap$ is an isomorphism in all degrees in this
case, too.

\subsubsection{Some~$C$-fields}
When~$F$ is a~$C$-field of level~$1$, the~$W$-group~$\G$ is~$\prod_{i
  \in I} \z/4$, where~$I$ is a (possibly infinite) basis of~$F^\times/
(F^\times)^2$. When~$I$ has order~$n$, then the ring~$H^*(\G)_{dec}$
is an exterior algebra on~$n$ generators in degree~$1$, so it is
concentrated in degrees~$\le n$. For~$n \le 2$, it follows
that~$\ourmap$ is an isomorphism. For~$n=3$, a computer calculation
which we will not reproduce here shows that~$\ourmap$ is again an
isomorphism. It remains an open problem to compute the graded
representation ring of~$(\z/4)^n$, and the character~$\cha$ is of no
help here: one can check that~$I_{(\z/4)^n} = H^{\ge 3}((\z/4)^n)$.

\subsubsection{Any field whose~$W$-group is~$D_4$} This is the case of
the field~$F= \r(\!(u)\!)(\!(v)\!)$ for example. From
Proposition~\ref{prop-graded-D4}, it follows that~$\ourmap$ is an
isomorphism. 

\subsubsection{Any field with a universal~$W$-group}
Given any set~$I$, there is a universal~$W$-group which we will write
here~$\G_I$. When~$F$ is a field such that~$F^\times / (F^\times)^2$
has a basis in bijection with~$I$, then its~$W$-group~$\G$ is a
quotient of~$\G_I$. It may happen that~$\G = \G_I$ (see~\cite{gao} for
examples). In this case, from the fact that~$H^2(\G_I) = 0$, we see
that the Galois cohomology is concentrated in degrees~$\le 1$,
so~$\ourmap$ is certainly an isomorphism.

\subsubsection{Any field for which~$\ell(-1)$ is not a zero divisor}
To treat this case, we begin by establishing a simple formula for the
action of the operations~$\theta_n$ in the case of Galois
cohomology. The formula was anticipated in the Introduction. 

\begin{lem} \label{lem-theta-n-for-fields}
Let~$G$ be the absolute Galois group of a field~$F$, and let~$x \in
H^n(G)$. Then~$\theta_n(x) = \ell(-1)^{2^{n-1} - n} x$.
\end{lem}

\begin{proof}
  Since~$H^*(G)$ is generated by elements of degree~$1$, it is enough
  to prove this for~$x = t_1 t_2 \cdots t_n$ with~$t_i \in H^1(G)$. In
  this case the value of~$\theta_n(t_1 \cdots t_n)$ is given by
  Lemma~\ref{lem-key}. However, in Galois cohomology one has~$t^2 =
  \ell(-1) t$ for any~$t$ of degree~$1$. It follows that there exists
  a constant~$c_n$, depending only on~$n$, such that~$\theta_n(x) =
  c_n \ell(-1)^{2^{n-1} - n}x$ for any~$x$ of this form (namely,~$c_n$
  is the number of terms in the sum appearing in the statement of
  Lemma~\ref{lem-key}, reduced mod~$2$).

It remains to see that~$c_n \ne 0$. Indeed if it were~$0$, then one
would have~$\theta_n(x) = 0$ for any~$x \in H^n(G)$, regardless
of~$G$. However the example of~$\z/2$ (which is the absolute Galois group
of~$\r$) shows that~$\theta_n x$ can be non-zero, since
Proposition~\ref{prop-ideal-elementary-abelian} establishes
that~$\id_{\z/2} = (0)$.
\end{proof}

The commutative square of Theorem~\ref{thm-map-K-to-grW} then shows
clearly that~$\ourmap$ is an isomorphism when~$\ell(-1)$ is not a zero
divisor.

%% file: main.bbl
\newcommand{\noopsort}[1]{} \newcommand{\printfirst}[2]{#1}
  \newcommand{\singleletter}[1]{#1} \newcommand{\switchargs}[2]{#2#1}
\providecommand{\bysame}{\leavevmode\hbox to3em{\hrulefill}\thinspace}
\providecommand{\MR}{\relax\ifhmode\unskip\space\fi MR }
\providecommand{\MRhref}[2]{%
  \href{http://www.ams.org/mathscinet-getitem?mr=#1}{#2}
}
\providecommand{\href}[2]{#2}
\begin{thebibliography}{AKM99}

\bibitem[Ada92]{adams}
J.~Frank Adams, \emph{Two theorems of {J}. {Lannes}, the selected works of {J}.
  {F}rank {A}dams. {V}ol. {II}}, Cambridge University Press, Cambridge, 1992.
  \MR{1203313 (94d:01095b)}

\bibitem[AKM99]{ademminac}
Alejandro Adem, Dikran~B. Karagueuzian, and J{\'a}n Min{\'a}{\v{c}}, \emph{On
  the cohomology of {G}alois groups determined by {W}itt rings}, Adv. Math.
  \textbf{148} (1999), no.~1, 105--160. \MR{1736643 (2001b:12010)}

\bibitem[AS69]{atiyahsegal}
M.~F. Atiyah and G.~B. Segal, \emph{Equivariant {$K$}-theory and completion},
  J. Differential Geometry \textbf{3} (1969), 1--18. \MR{0259946 (41 \#4575)}

\bibitem[AT69]{atiyahtall}
M.~F. Atiyah and D.~O. Tall, \emph{Group representations, {$\lambda $}-rings
  and the {$J$}-homomorphism}, Topology \textbf{8} (1969), 253--297.
  \MR{0244387 (39 \#5702)}

\bibitem[BF91]{franjou}
D.~J. Benson and V.~Franjou, \emph{S\'eries de compositions de modules
  instables et injectivit\'e de la cohomologie du groupe {${\bf Z}/2$}}, Math.
  Z. \textbf{208} (1991), no.~3, 389--399. \MR{1134584 (93b:55022)}

\bibitem[EM11]{minacefrat}
Ido Efrat and J{\'a}n Min{\'a}{\v{c}}, \emph{On the descending central series
  of absolute {G}alois groups}, Amer. J. Math. \textbf{133} (2011), 1503--1532.

\bibitem[FL85]{fultonlang}
William Fulton and Serge Lang, \emph{Riemann-{R}och algebra}, Grundlehren der
  Mathematischen Wissenschaften [Fundamental Principles of Mathematical
  Sciences], vol. 277, Springer-Verlag, New York, 1985. \MR{801033 (88h:14011)}

\bibitem[GM97]{gao}
Wenfeng Gao and J{\'a}n Min{\'a}{\v{c}}, \emph{Milnor's conjecture and {G}alois
  theory. {I}}, Algebraic {$K$}-theory ({T}oronto, {ON}, 1996), Fields Inst.
  Commun., vol.~16, Amer. Math. Soc., Providence, RI, 1997, pp.~95--110.
  \MR{1466972 (98h:12006)}

\bibitem[Kuh94]{kuhn}
Nicholas~J. Kuhn, \emph{Generic representations of the finite general linear
  groups and the {S}teenrod algebra. {I}}, Amer. J. Math. \textbf{116} (1994),
  no.~2, 327--360. \MR{1269607 (95c:55022)}

\bibitem[Kum75]{kummer}
Ernst~Eduard Kummer, \emph{Collected papers}, Springer-Verlag, Berlin, 1975,
  Volume I: Contributions to number theory. \MR{0465760 (57 \#5650a)}

\bibitem[Mil58]{milnorsteenrod}
John Milnor, \emph{The {S}teenrod algebra and its dual}, Ann. of Math. (2)
  \textbf{67} (1958), 150--171. \MR{0099653 (20 \#6092)}

\bibitem[Mil70]{milnorconjecture}
\bysame, \emph{Algebraic {$K$}-theory and quadratic forms}, Invent. Math.
  \textbf{9} (1969/1970), 318--344. \MR{0260844 (41 \#5465)}

\bibitem[MS74]{milnorchar}
John~W. Milnor and James~D. Stasheff, \emph{Characteristic classes}, Princeton
  University Press, Princeton, N. J., 1974, Annals of Mathematics Studies, No.
  76. \MR{0440554 (55 \#13428)}

\bibitem[MS96]{minacspira}
J{\'a}n Min{\'a}{\v{c}} and Michel Spira, \emph{Witt rings and {G}alois
  groups}, Ann. of Math. (2) \textbf{144} (1996), no.~1, 35--60. \MR{1405942
  (97i:11038)}

\bibitem[OVV07]{vovwitt}
D.~Orlov, A.~Vishik, and V.~Voevodsky, \emph{An exact sequence for {$K\sp M\sb
  \ast/2$} with applications to quadratic forms}, Ann. of Math. (2)
  \textbf{165} (2007), no.~1, 1--13. \MR{2276765 (2008c:19001)}

\bibitem[Voe03a]{vov2}
Vladimir Voevodsky, \emph{Motivic cohomology with {${\bf Z}/2$}-coefficients},
  Publ. Math. Inst. Hautes \'Etudes Sci. (2003), no.~98, 59--104. \MR{2031199
  (2005b:14038b)}

\bibitem[Voe03b]{vov1}
\bysame, \emph{Reduced power operations in motivic cohomology}, Publ. Math.
  Inst. Hautes \'Etudes Sci. (2003), no.~98, 1--57. \MR{2031198 (2005b:14038a)}

\end{thebibliography}
